\documentclass[12pt,centertags]{amsart}
\usepackage[usenames,dvipsnames]{xcolor}
\usepackage[colorlinks,linkcolor=blue,citecolor=blue,bookmarksnumbered,final]{hyperref}
\usepackage{graphicx}

\usepackage{amsmath,amstext,amsthm,a4,amssymb,amscd,mathrsfs}
\usepackage{mathtools}
\usepackage[mathscr]{eucal}
\usepackage{mathrsfs}
\usepackage{epsf}
\numberwithin{equation}{section}

\def\mA{\mathcal{A}}

\def\mE{\mathcal{E}}
\def\mF{\mathcal{F}}
\def\mH{\mathcal{H}}

\def\mL{\mathcal{L}}
\def\mM{\mathcal{M}}
\def\mN{\mathcal{N}}

\DeclareMathOperator{\supp}{Supp}

\newcommand{\dd}{\mathrm{d}}

\newtheorem{thm}{Theorem}[section]
\newtheorem{lemma}[thm]{Lemma}
\newtheorem{prop}[thm]{Proposition}

\theoremstyle{definition}

\theoremstyle{definition}

\theoremstyle{definition}
\newtheorem{defn}[thm]{Definition}
\newcommand{\be}{\begin{eqnarray}}
\newcommand{\ee}{\end{eqnarray}}

\newcommand{\comment}[1]{}

\colorlet{RED}{red}

\usepackage{tikz}
\usetikzlibrary{shapes,arrows,patterns,calc,decorations.pathmorphing}
\tikzstyle{block} = [rectangle, draw, fill=cyan!20, text centered, rounded corners, minimum height=2em]
\tikzstyle{line} = [draw, -latex']

\begin{document}

\title[Nonnegative scalar curvature and area decreasing maps]{Nonnegative scalar curvature and area decreasing maps on complete foliated manifolds}
 
\author{Guangxiang Su, Xiangsheng Wang and Weiping Zhang}
\address{Chern Institute of Mathematics \& LPMC, Nankai
University, Tianjin 300071, P.R. China}
\email{guangxiangsu@nankai.edu.cn}
\address{School of Mathematics, Shandong University, Jinan, Shandong 250100, P.R. China}
\email{xiangsheng@sdu.edu.cn}
\address{Chern Institute of Mathematics \& LPMC, Nankai
University, Tianjin 300071, P.R. China}
\email{weiping@nankai.edu.cn}

\begin{abstract}  
  Let $(M,g^{TM})$ be a noncompact complete Riemannian manifold of dimension $n$, {and} let $F\subseteq TM$ be an integrable subbundle of $TM$. Let $g^F=g^{TM}|_{F}$ be the restricted metric on $F$ and let $k^F$ be the associated leafwise scalar curvature. Let $f:M\to S^n(1)$ be a smooth area decreasing map along $F$, which is locally constant near infinity and of non-zero degree. We show that if $k^F> {\rm rk}(F)({\rm rk}(F)-1)$ on the support of ${\rm d}f$, and either $TM$ or $F$ is spin, then $\inf (k^F)<0$. As a consequence, we prove Gromov's sharp foliated $\otimes_\varepsilon$-twisting conjecture. {Using the same method,} we also extend {two famous non-existence results due to Gromov and Lawson about $\Lambda^2$-enlargeable metrics (and/or manifolds) to the foliated case.}

\end{abstract}

\maketitle

\section{Introduction} \label{s0}

In this paper, we always assume that $M$ is a smooth connected oriented manifold without boundary.
{We call a pair $(M,F)$ a foliated manifold if $F$ is a foliation of $M$, or equivalently, an integrable subbundle of $TM$.}

\subsection{An extension of Llarull's theorem}
 It is well known that starting with the famous Lichnerowicz vanishing theorem \cite{L63}, Dirac operators have played important roles in the study of Riemannian metrics of positive scalar curvature on spin manifolds (cf. \cite{GL83}, \cite{LaMi89}). A notable example is Llarull's rigidity theorem \cite{LL} which states that for a compact spin Riemannian manifold $(M,g^{TM})$ of dimension $n$ such that the associated scalar curvature 
$k^{TM}$ verifies that $k^{TM}\geq n(n-1)$, then any (non-strictly) area decreasing smooth map $f:M\to S^n(1)$ of non-zero degree is an isometry.

In answering a question of Gromov in an earlier version of \cite{Gr}, Zhang in \cite{Z19-2} proves that for an even dimensional noncompact complete spin Riemannian manifold $(M,g^{TM})$ and a smooth (non-strictly) area decreasing map $f:M\to S^{\dim M}(1)$ which is locally constant near infinity and of non-zero degree, if the associated scalar curvature $k^{TM}$ verifies
\begin{equation*}
  k^{TM}\geq (\dim M) (\dim M-1)\text{ on }{\rm Supp}({\rm d}f),
\end{equation*}
then $\inf (k^{TM})<0$.

The main idea in \cite{Z19-2}, which goes back to \cite[(1.11)]{Z19}, is to deform the involved twisted Dirac operator on $M$ by a suitable endomorphism of the twisted vector bundle. {Since the deformed Dirac operator is invertible near infinity, one can apply} the relative index theorem to {obtain a contradiction}.

In this paper, we generalize \cite{Z19-2} to the {foliated case}. 
For this purpose, inspired by~\cite[Definition 6.1]{GL83} and~\cite[Definition~0.1]{Z19}, we define the following class of maps.
\begin{defn}
  Let $(X,\mathcal{L})$ be a foliated manifold.
  A $C^1$-map $\varphi:X\to Y$ between Riemannian manifolds is said to be $(\epsilon,\Lambda^2)$-contracting along $\mathcal{L}$, if for all $x\in X$, the map $\varphi_*:\Lambda^2 (T_x X)\to \Lambda^2 (T_{\varphi(x)}Y)$ satisfies
  \begin{equation*}
    |\varphi_*(V_x\wedge W_x)| \le \epsilon |V_x\wedge W_x|,
  \end{equation*}
  for any $V_x, W_x\in \mathcal{L}_x\subseteq T_xX$.
\end{defn}

If $\mathcal{L} = TX$, the above definition coincides with the usual definition of the $(\epsilon,\Lambda^2)$-contracting map in~\cite[Definition 6.1]{GL83}.
Moreover, similar to the definition of the area decreasing map, if a map $\varphi$ is $(1,\Lambda^2)$-contracting along $\mathcal{L}$, we call $\varphi$ area decreasing along $\mathcal{L}$.

{Let $(M,F)$ be a noncompact foliated manifold of dimension $n$.}
Let {$g^{TM}$ be a complete Riemannian metric of $M$,} let $g^F=g^{TM}|_{F}$ be the restricted metric on $F$ and let $k^F$ be the associated leafwise scalar curvature. Let $f:M\to S^n(1)$ be a smooth map, which is area decreasing along $F$, and is locally constant near infinity and of non-zero degree. Let ${\rm d}f:TM\to TS^n (1)$ be the differential of $f$. The support of ${\rm d}f$ is defined to be ${\rm Supp}({\rm d}f)=\overline{\{x\in M:{\rm d}f(x)\neq 0\}}$.

The main result of this paper can be stated as follows.

\begin{thm}\label{t1.1}
Under the above assumptions, if either $TM$ or $F$ is spin and
\begin{align}\label{1.2}
k^F>{\rm rk}(F)({\rm rk}(F)-1) \ {\rm on}\ {\rm Supp}({\rm d}f),
\end{align}
then one has
\begin{align}\label{1.3}
\inf (k^F)<0. 
\end{align}
\end{thm}

As an application of Theorem~\ref{t1.1}, we resolve the following conjecture due to Gromov, {appeared in the fourth version of his four lectures,}~\cite[p. 61]{Gr}.

\vspace{.6\baselineskip}
\noindent{\bf Sharp Foliated $\otimes_\varepsilon$-Twisting Conjecture}. {\it Let $X$ be a complete oriented $n$-dimensional Riemannian manifold with a smooth $m$-dimensional, $2\leq m\leq n$, spin foliation $\mathcal{L}$, such that the induced Riemannian metrics on the leaves of ${\mL}$ have their scalar curvatures $>m(m-1)$. Then $X$ admits no smooth area decreasing locally constant at infinity map $f:X\to S^n$ with ${\rm deg}(f)\neq 0$.} 
\vspace{\topsep}

\begin{proof}
  Since an area decreasing map on $X$ is area decreasing along $\mathcal{L}$ automatically, as a consequence of Theorem \ref{t1.1}, Gromov's above conjecture holds.
\end{proof}

{Note that Theorem~\ref{t1.1} also implies that if} the $\mathcal{L}$ spin condition {is} replaced by the $X$ spin condition, {the above conjecture still holds}.

{On the other direction, in~\cite{S}, Su proves a generalization of Llarull's theorem for the foliated compact manifolds. Therefore, Theorem~\ref{t1.1} can also be viewed as a noncompact extension of~\cite{S}. Figure~\ref{fig:1} gives an illustration about the relation between several results.}

\begin{figure}[htbp]
  \centering
  \begin{tikzpicture}[node distance = 2cm, auto]
      \node [block] (init) {Llarull's theorem, compact $M$, $TM$};
      \node [block, below left of=init, xshift=-6em] (l1) {\cite{Gr,Z19-2}, noncompact $M$, $TM$};
      \node [block, below right of=init, xshift=6em] (r1) {\cite{S}, compact $M$, $F$};
      \node [block, below right of=l1, xshift=6em] (fin) {Theorem~\ref{t1.1}, noncompact $M$, $F$};

      \path [line] (init.west) -|  (l1.north);
      \path [line] (init.east) -|  (r1.north);
      \path [line] (l1.south) |-  (fin.west);
      \path [line] (r1.south) |-  (fin.east);

    \end{tikzpicture}
    \caption{The relation between several results.}
    \label{fig:1}
\end{figure}
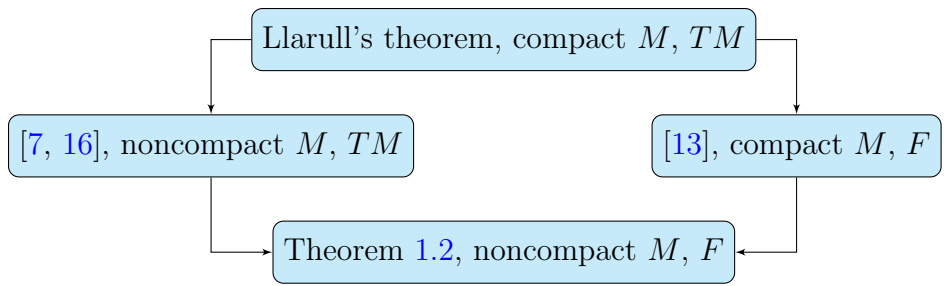

{To compare Theorem~\ref{t1.1} with the related results in the literature further,
we recall that for the enlargeable foliated noncompact manifold,} in \cite{SZ}, Su and Zhang also show a similar estimate on the leafwise scalar curvature. More precisely,
{let $(M,F)$ be a foliated manifold carrying a (not necessarily complete) Riemannian metric $g^{TM}$}. Let $k^F$ be the leafwise scalar curvature associated to $g^F=g^{TM}|_{F}$. {If either $TM$ or $F$ is spin and $g^{TM}$ is enlargeable}, then $\inf(k^F)\leq 0$.
{In fact, the argument in~\cite{SZ} motivates our proof of Theorem~\ref{t1.1} partially.}

We will put the different kinds of deformations of Dirac operators appeared in \cite{SZ}, \cite{Z17} and \cite{Z19-2} together (cf. (\ref{0.15})) to prove Theorem \ref{t1.1}. Still, the sub-Dirac operators constructed in \cite{LZ01} and \cite{Z17}, as well as   the Connes fibration introduced in \cite{Co86}  (cf. \cite{Gr}, \cite[\S 2.1]{Z17}), will play  essential roles in our proof. But the new difficulty in the current case is that the map $f$ is area decreasing {along $F$}, i.e.,\ $f$ only contracts on two {forms in some sense}\footnote{As a contrast, $f$ contracts on one {forms} in~\cite{SZ}.}. Such a weaker assumption on $f$ forces us to construct new cut-off functions to replace $\varphi_{\varepsilon,R,1}$ and $\varphi_{\varepsilon,R,2}$ in \cite[(1.23)]{SZ}.

Recall that in \cite[Theorem 1.17]{GL83}, Gromov and Lawson use a small perturbation of the distance function to prove their relative index theorem. We will adapt this perturbed distance function to construct the cut-off functions needed in the current case. As a result, unlike~\cite{SZ}, the completeness of the manifold is necessary in our proof. As in \cite{Z17}, we only give the proof of Theorem \ref{t1.1} for the $TM$ spin case in detail. The $F$ spin case can be proved similarly as in \cite[\S 2.5]{Z17}.

\subsection{Two non-existence results}
It turns out that our method to prove Theorem \ref{t1.1} can also be used to generalize several classical results about scalar curvature to the {foliated case}.

In \cite{GL83}, Gromov and Lawson introduce the {concept of $\Lambda^2$-enlargeability}. In the foliated case, we use the following variant of \cite[Definition 7.1]{GL83}.

\begin{defn}\label{cc1.3}
A Riemannian metric on a connected foliated manifold $(M,F)$ is called $\Lambda^2$-enlargeable along $F$ if given any $\epsilon>0$, there exist a covering manifold ${M}_\epsilon\to M$ such that either $M_\epsilon$ or $F_\epsilon$ (the lifted foliation of $F$ in $M_\epsilon$) is spin and a smooth map $f_\epsilon:M_\epsilon\to S^{\dim M}(1)$ which is $(\epsilon,\Lambda^2)$-contracting along $F_\epsilon$ (with respect to the lifted metric), constant near infinity and of non-zero degree.
\end{defn}

Gromov and Lawson use the $\Lambda^2$-enlargeable metrics to define the $\Lambda^2$-enlargeable manifolds. We adapt their definition \cite[Definition 6.4]{GL83} to our situation as follows.

 \begin{defn}
 A connected (not necessarily compact) foliated manifold {$(M,F)$} is said to be $\Lambda^2$-enlargeable along $F$ if any Riemannian metric (not necessarily complete!) on {$M$} is $\Lambda^2$-enlargeable along {$F$}.
\end{defn}

As before, if {$F=TM$}, the above two definitions coincide with the usual definition of the $\Lambda^2$-enlargeable metric or manifold.

In \cite{GL83}, Gromov and Lawson prove the following famous theorem about $\Lambda^2$-enlargeable metrics.
{Recall that a function on a manifold is called uniformly positive if the infimum of this function is strictly positive.}
\begin{thm}{\rm (Gromov-Lawson, \cite[Theorem 7.3]{GL83})}\label{th1.6}
No complete Riemannian metric which is $\Lambda^2$-enlargeable can have uniformly positive scalar curvature.
\end{thm}

By the same argument in the proof of Theorem \ref{t1.1}, we can extend Theorem \ref{th1.6} as follows. 
\begin{thm}\label{th6}
{Let $(M,F)$ be a foliated manifold. For any complete Riemannian metric $g^{TM}$ on $M$ which is $\Lambda^2$-enlargeable along $F$, $k^F$, the leafwise scalar curvature of $g^{TM}$ along $F$, cannot be uniformly positive.}
\end{thm}

{If we further assume that $M$ itself is $\Lambda^2$-enlargeable, Gromov and Lawson in {\cite{GL83}}} prove the following famous theorem, {which strengthens the result of Theorem~\ref{th1.6} in the following way.}
 \begin{thm}{\rm (Gromov-Lawson, \cite[Theorem 6.12]{GL83})}\label{th1.4}
 A manifold $M$ which is $\Lambda^2$-enlargeable, cannot carry a complete metric of positive scalar curvature. 
 \end{thm}

The following result is a foliated extension of Theorem \ref{th1.4}.

 \begin{thm}\label{t5}
{Let $(M,F)$ be a foliated manifold}. If $M$ is $\Lambda^2$-enlargeable along $F$, then $M$ cannot carry a complete metric $g^{TM}$ {satisfying that} $k^F$, the leafwise scalar curvature of $g^{TM}$ along $F$, is positive everywhere.
\end{thm}

Theorem \ref{th6} and Theorem \ref{t5} extend \cite[Theorem 1.7]{BH19} and \cite[Theorem 0.2]{Z19} to the noncompact situation. We would like to mention that Benameur and Heitsch \cite{BH20} have  also studied non-existence of positive scalar curvature metrics on noncompact foliated manifolds.

The rest of the paper is organized as follows. In Section~\ref{s2}, we prove Theorem \ref{t1.1} for the even dimensional case. In Section~\ref{sec:odd}, we prove Theorem \ref{t1.1} for the odd dimensional case. In Section~\ref{sec:th6}, we prove Theorem \ref{th6}. In Section~\ref{sec:th5}, we prove Theorem \ref{t5}.

\section{Proof of Theorem \ref{t1.1}: the even dimensional case}\label{s2}

\setcounter{equation}{0}

In this section, we prove Theorem \ref{t1.1} for the even dimensional case. In fact, we first show this theorem under an additional assumption that $f$ is constant near the infinity in Subsections~\ref{s1.1}-\ref{s1.4}. We hope that in this case, the idea behind the proof is easier to undertand. More specifically, in Subsection \ref{s1.1}, we recall the basic geometric setup. In Subsection \ref{s1.2}, we review the definition of the Connes fibration and explain how to lift the geometric data to the Connes fibration. In Subsection \ref{s1.3}, we study the deformed sub-Dirac operators on the Connes fibration. In Subsection \ref{s1.4}, we finish the proof of Theorem \ref{t1.1} for the even dimensional case with the above simplified assumption. In Subsection~\ref{sub:locally-const-near}, we discuss the modifications needed for the general case.

\subsection{The basic geometric setup}\label{s1.1}
Let $M$ be a noncompact even dimensional Riemannian manifold of dimension $n$ carrying a {complete} Riemannian metric $g^{TM}$ and $F$ an integrable subbundle of the tangent bundle $TM$. Let $S^n(1)$ be the standard $n$-dimensional unit sphere carrying its canonical metric. {As explained in Introduction, we further assume that $TM$ is spin.}

We now assume that $f:M\to S^n(1)$ is a smooth map, which is area decreasing along $F$. Except in the last subsection of this section, we also assume that $f$ is constant near infinity\footnote{That is, $f$ is a constant map outside a compact subset of $M$.} satisfying
\begin{align}\label{2.2}
{\rm deg}(f)\neq 0.
\end{align}

Let $\dd f:TM\to TS^n (1)$ be the differential of $f$. The support of $\dd f$ is defined to be ${\rm Supp}(\dd f)=\overline{\{x\in M:\dd f(x)\neq 0\}}$.

Let $g^F=g^{TM}|_F$ be the induced Euclidean metric on $F$. Let $k^F\in C^\infty(M)$ be the  leafwise scalar curvature associated to $g^F$ (cf. \cite[(0.1)]{Z17}). To show Theorem~\ref{t1.1}, we argue by contradiction. Assume that (\ref{1.3}) does not hold, that is, 
\begin{equation}
  \label{eq:kf-lb}
\inf(k^F)\geq 0.
\end{equation}

Let $F^\perp$ be the orthogonal complement to $F$, i.e., we have the orthogonal splitting 
 \begin{align}\label{0.3}
TM=F\oplus F^\perp,\ \ \ g^{TM}=g^F\oplus g^{F^\perp}.
\end{align}

Following \cite[Theorem 1.17]{GL83}, we choose a fixed point $x_0\in M$ and let $d:M\to \mathbb{R}^+$ be a regularization of the distance function ${\rm dist}(x,x_0)$ such that
{
  \begin{equation}
    \label{eq:fun-d}
  |\nabla d|(x)\leq 3/2,
\end{equation}
for any $x\in M$.}

Set
\begin{equation*}\label{728}
  B_m=\{x\in M: d(x)\leq m\},\ m\in \mathbb{N}.
\end{equation*}
Since the Riemannian metric $g^{TM}$ is complete, $B_m$ is compact.\footnote{In fact, if there is a smooth function $l$ on $M$ satisfying properties similar to $d$, that is, $|\nabla l|$ is bounded and $l^{-1}((-\infty, m])$, $m\in \mathbf{N}$, is compact, by~\cite{GW}, $M$ must be a complete manifold.}

Let $K\subseteq M$ be a compact subset such that $f$ is constant outside $K$, that is,
\begin{equation*}
  f(M\setminus K)=z_0\in S^n(1).
\end{equation*}
Since $K$ is compact, we can choose a sufficiently large $m$ such that $K\subseteq B_m$. This implies
\begin{equation}
  \label{eq:df-b}
  \supp(\dd f) \subseteq K \subseteq B_m.
\end{equation}

Following \cite{GL83}, we take a compact hypersurface $H_{3m}\subseteq M\setminus B_{3m}$, cutting $M$ into two parts such that the compact part, denoted by $M_{H_{3m}}$, contains $B_{3m}$. Then $M_{H_{3m}}$ is a compact smooth manifold with boundary $H_{3m}$.

To make the gluing process in due course easy to understand, we deform the metric near $H_{3m}$ a little bit as follows. Let $g^{TH_{3m}}$ be the induced metric on $H_{3m}$.
On the product manifold $H_{3m} \times [-1,2]$, we construct a metric as follows.
Near the boundary $H_{3m}\times \{-1\}$ of $H_{3m} \times [-1, 2]$, that is, $H_{3m} \times [-1, -1+ \varepsilon')$, by using the geodesic normal coordinate of $H_{3m} \subseteq M_{H_{3m}}$, we can identify $H_{3m} \times [-1, -1+ \varepsilon')$ with a neighborhood of $H_{3m}$, {denoted by} $U$, in $M_{H_{3m}}$ via a diffeomorphism $\iota$.
Now, we require the metric on $H_{3m} \times [-1, -1+\varepsilon')$ to be the pull-back metric obtained from that of $U$ by $\iota$.
In the same way, we can construct a metric near the boundary $H_{3m}\times \{2\}$ of $H_{3m} \times [-1, 2]$, i.e.,\ $H_{3m} \times (2-\varepsilon'', 2]$.
Meanwhile, on $H_{3m}\times [0,1]$, we give the product metric {constructed} by $g^{TH_{3m}}$ and the standard metric on $[0,1]$. 
Finally, the metric on $H_{3m} \times [-1,2]$ is a smooth extension of the metrics on the above three pieces.

Let $M'_{H_{3m}}$ be another copy of $M_{H_{3m}}$ with the same metric and the opposite orientation.
Let $\iota'$ be the diffeomorphism, the isometry actually, from $H_{3m} \times (2-\varepsilon'', 2]$ to a neighborhood of $\partial M'_{H_{3m}}$, $U'$, in $M'_{H_{3m}}$.
On the disjoint union,
\begin{equation*}
  M^{\circ}_{H_{3m}} \sqcup H_{3m} \times (-1,2) \sqcup M'^{,\circ}_{H_{3m}},
\end{equation*}
we consider the equivalent relation $\sim$ given by $x_1 \sim x_2$ if and only if $x_1\in U^{\circ}$, $x_2\in H_{3m} \times (-1, -1+\varepsilon')$ (resp.\ $x_1\in U'^{,\circ}$, $x_2\in H_{3m} \times (2-\varepsilon'', 2)$) and $x_1= \iota(x_2)$ (resp.\ $x_1 = \iota'(x_2)$).
As a set, we define the gluing manifold $\widehat{M}_{H_{3m}}$ to be
\begin{equation*}
  \widehat{M}_{H_{3m}} = (M^{\circ}_{H_{3m}} \sqcup H_{3m} \times (-1,2) \sqcup M'^{,\circ}_{H_{3m}})/\sim,
\end{equation*}
endowed with the differentiable structure associated with the open cover
\begin{equation*}
  \{M^{\circ}_{H_{3m}}, H_{3m} \allowbreak\times (-1,2), M'^{,\circ}_{H_{3m}}\}.
\end{equation*}
Moreover, since $\iota$ and $\iota'$ are isometries with respect to the metrics on
\begin{equation*}
  \{M^{\circ}_{H_{3m}}, H_{3m} \times (-1,2), M'^{,\circ}_{H_{3m}}\},
\end{equation*}
$\widehat{M}_{H_{3m}}$ also inherits a metric from this open cover.
From now on, we view $M_{H_{3m}}$, $M'_{H_{3m}}$ and $H_{3m}\times [-1,2]$ as submanifolds of $\widehat{M}_{H_{3m}}$.

{Figure~\ref{fig:tp} helps to explain this gluing procedure.}

\begin{figure}[ht]
  \centering
  \includegraphics{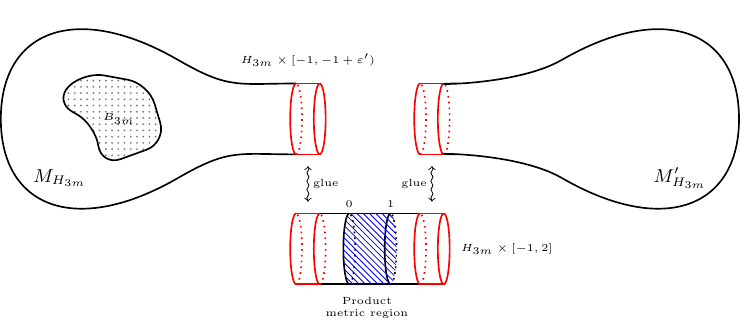}
  \caption{Gluing three parts.}
  \label{fig:tp}
\end{figure}

\subsection{The Connes fibration}\label{s1.2}
By following \cite[\S 5]{Co86} (cf. \cite[\S 2.1]{Z17}), let $ \pi:\mM\rightarrow M$ be the Connes
fibration over $M$ such that for any $x\in M$, $\mM_{x}=\pi^{-1}(x)$
is the space of Euclidean metrics on the linear space $T_xM/F_{x}$.
Let  $T^V\mM$ denote the vertical tangent bundle of the fibration
$\pi:\mM\rightarrow M$. Then it carries a natural metric
$g^{T^V\mM}$ such that   any two points $p,\,q\in \mM_{x}$ with $x\in M$ can be joined by a unique geodesic along $\mM_{x}$. Let $d^{\mM_{x}}(p,q)$ denote the length of this geodesic.  
  
By  using the Bott connection
  on $TM/F$ (cf. \cite[(1.2)]{Z17}), which is leafwise flat, one  lifts $F$ to an integrable subbundle
$\mF$ of $T\mM$. 
  Then $g^{F}$   lifts to a Euclidean metric $g^{\mF}= \pi^* g^{F} $ on $\mF$.

  Let $\mF_{1}^\perp\subseteq T\mM$ be a subbundle, which is  transversal to $\mF\oplus T^V\mM$,   such that we have a
splitting $T\mM=(\mF \oplus T^V \mM)\oplus\mF_{1}^\perp$. Then
$\mF_{1}^\perp$ can be identified with $T\mM/(\mF \oplus T^V \mM)$
and carries a canonically induced metric $g^{\mF_{1}^\perp}$.
We denote  $\mF_{2}^\perp$ to be $T^V\mM$.

  The metric $g^{F^\perp}$   in (\ref{0.3}) determines a canonical embedded section $s: M\hookrightarrow \mM$.
For any $p\in\mM$, set 
\begin{equation*}
  \rho(p)=d^{\mM_{\pi(p)}}(p,s(\pi(p) )).
\end{equation*}

For any $  \beta,\ \gamma>0$,  following  \cite[(2.15)]{Z17}, let $g_{\beta,\gamma}^{T\mM}$ be the   metric    on $T\mM$  defined by
the
  orthogonal splitting,
\begin{align}\label{0.6}\begin{split}
       T\mM =   \mF\oplus \mF^\perp_{1}\oplus \mF^\perp_{2},  \ 
\  \  \
g^{T\mM}_{\beta,\gamma}= \beta^2   g^{\mF}\oplus\frac{
g^{\mF^\perp_{1}}}{ \gamma^2 }\oplus g^{\mF^\perp_{2}}.\end{split}
\end{align}

For any $R>0$, let $ \mM_{R}$ be the smooth manifold with boundary
defined by
\begin{align*}
\mM_{R}=\left\{ p\in \mM\ :\  \rho(p)\leq R \right\}.
\end{align*}

Set $\mH_{3m}= \pi^{-1}(H_{3m})$ and 
\begin{align*}\label{0.8} 
\mM_{\mH_{3m},R} = \left(\pi^{-1}
\left(M_{H_{3m}}\right)\right)\cap \mM_{R},\; { \mH_{3m,R} = \mH_{3m}\cap \mM_{R}}.
\end{align*}

{Consider another copy $\mM_{\mH_{3m},R} '$ of $\mM_{\mH_{3m},R} $ carrying the metric $g^{T\mM'_{3m,R}}$ defined by equation (\ref{0.6}) with $\beta=\gamma=1$.} We glue $\mM_{\mH_{3m},R} $, $\mM_{\mH_{3m},R} '$ and $\mH_{3m,R}\times [-1,2]$ together to get a manifold $\widehat{\mM}_{\mH_{3m},R}$ as we have done for $\widehat{M}_{H_{3m}}$. The difference is that $\widehat{\mM}_{\mH_{3m},R}$ is a smooth manifold with boundary. To write the boundary manifold explicitly, we note that the boundary of $\mM_{\mH_{3m},R} $ consists of two smooth pieces of top dimension: { one is $\mH_{3m,R}$, another is denoted by $\mA$. Note that
  \begin{equation*}
    \pi(\mA) = M_{H_{3m}} \setminus H_{3m}.
  \end{equation*}
For $\mM_{\mH_{3m},R}' $, we can find a similar boundary piece $\mA'$. Then $\partial \widehat{\mM}_{\mH_{3m},R}$ is the closed manifold glued together by $\mA$, $\mA'$ and $\partial \mH_{3m,R} \times [-1, 2]$.} Without loss of generality, we assume that $\widehat\mM_{\mH_{3m},R}$ is oriented.

\begin{figure}[htbp]
  \centering
  \includegraphics[scale=1]{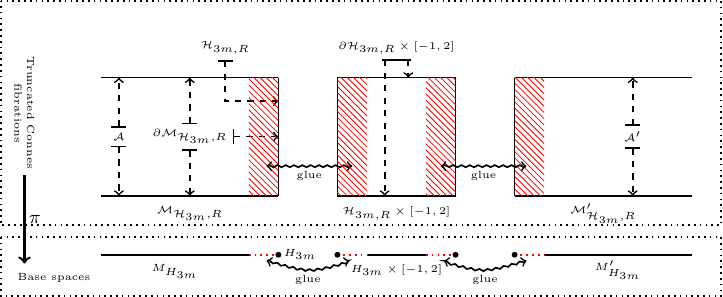}
  \caption{Gluing the truncated Connes fibration.}
  \label{fig:cf}
\end{figure}

{Figure~\ref{fig:cf} is a heuristic illustration about the different pieces of $\partial\mM_{\mH_{3m},R} $ and how to glue three truncated Connes fibrations together.}

{Let $g^{T\mH_{3m,R}}$ be the induced metric on $\mH_{3m,R}$ by equation (\ref{0.6}) with $\beta=\gamma=1$ and let ${\rm d}t^2$ be the standard metric on $[0,1]$.} By the construction of $\widehat{\mM}_{\mH_{3m},R}$, we can define a smooth metric $g^{T\widehat{\mM}_{\mH_{3m},R}}$ on $\widehat{\mM}_{\mH_{3m},R}$ in the following way:
\begin{equation}
  \label{eq:metric}
  \begin{aligned}[]
    g^{T\widehat{\mM}_{\mH_{3m},R}}|_{\mM_{\mH_{3m},R}} &= g^{T\mM_{3m,R}}_{\beta,\gamma},\\
    g^{T\widehat{\mM}_{\mH_{3m},R}}|_{\mM'_{\mH_{3m},R}} &= g^{T\mM'_{3m,R}},\\
    g^{T\widehat{\mM}_{\mH_{3m},R}}|_{\mH_{3m,R}\times [0,1]} &= g^{T\mH_{3m,R}}\oplus {\dd t^2},
  \end{aligned}
\end{equation}
and then paste these metrics together.\footnote{We would like to point out that on $\widehat{\mathcal{M}}_{\mathcal{H}_{3m},R}$, the metric on $\mathcal{H}_{3m,R}\times [-1,0]$ also depends on $\beta,\gamma$. However, since we don't use the property of the metric on this part for the rest of the paper, we don't write it down explicitly.}

Let $\partial \widehat \mM_{\mH_{3m},R}$ bound another oriented manifold $  \mN_{3m,R}$ so that 
\begin{equation*}
  \widetilde  \mN_{3m,R}=\widehat \mM_{\mH_{3m},R}\allowbreak\cup \mN_{3m,R}
\end{equation*}
is an oriented closed manifold. 
Let $g_{}^{T\widetilde \mN_{3m,R}}$ be a smooth metric on $T\widetilde \mN_{3m,R}$ so that
\begin{equation*}
  g_{}^{T\widetilde \mN_{3m,R}}|_{\widehat\mM_{\mH_{3m},R}}=g_{}^{T\widehat\mM_{\mH_{3m},R} }.
\end{equation*}
The existence of $g_{}^{T\widetilde \mN_{3m,R}}$ is clear.

We extend $f:M_{H_{3m}}\to S^n(1)$ to $f:\widehat{M}_{H_{3m}}\to S^n(1)$ by setting
\begin{equation*}
  f\left((H_{3m}\times [-1,2])\cup M'_{H_{3m}}\right)=z_0.
\end{equation*}
 Let $\widehat f_{3m,R}:\widehat \mM_{\mH_{3m},R}\rightarrow S^{n}(1)$ be the smooth map defined by
\begin{align*}\label{0.9} 
\widehat f_{3m,R} = f\circ  \pi\ {\rm on}\ \mM_{\mH_{3m},R}
\end{align*}
and $\widehat f_{3m,R} \bigl(\left(\mH_{3m,R}\times [-1,2]\right)\cup \mM_{\mH_{3m},R}'\bigr)=z_0$.

Let $S(TS^n(1))=S_+(TS^n(1))\oplus S_-(TS^n(1))$ be the spinor bundle of $S^n(1)$.
Following~\cite[(2.6)]{Z19-2}, we construct a suitable bundle endomorphism $V$ of $S(TS^n(1))$.
More precisely, by taking any regular value $\mathfrak{q}\in S^n(1)\setminus f(M\setminus K)$ of $f$, we choose $X$ to be a smooth vector field on $S^n(1)$ such that $|X| > 0$ on $S^n(1)\setminus \{\mathfrak{q}\}$.
Let
\begin{equation*}
  v= c(X): S_+(TS^n(1))\to S_-(TS^n(1))
\end{equation*}
be the Clifford action of $X$ and
\begin{equation*}
  v^*: S_-(TS^n(1))\to S_+(TS^n(1))
\end{equation*}
be the adjoint of $v$ with respect to the Hermitian metrics on $S_{\pm}(TS^n(1))$.
We define $V$ to be the self-adjoint odd endomorphism $$V=v+v^*:S(TS^n(1))\to S(TS^n(1)).$$
Then there exists $\delta>0$ such that 
\begin{align}\label{2.11a}
\left(\widehat{f}_{3m,R}^*V\right)^2\geq \delta\ {\rm on}\ \widehat{\mM}_{\mH_{3m},R}\setminus \pi^{-1}\big({\rm Supp}(\dd f)\big).
\end{align}
Let
\begin{equation*}
  \left(\mE_{3m,R,\pm},g^{\mE_{3m,R,\pm}},\nabla^{\mE_{3m,R,\pm}}\right)= \widehat f_{3m,R}^*\left (S_{\pm}(TS^n(1)),g^{S_{\pm}(TS^n(1))},\nabla^{S_{\pm}(TS^n(1))}\right)
\end{equation*}
be the induced Hermitian vector bundle with the Hermitian connection on $\widehat \mM_{\mH_{3m},R}$. Then $\mE_{3m,R}=\mE_{3m,R,+}\oplus \mE_{3m,R,-}$ is a ${\mathbb{Z}}_2$-graded Hermitian vector bundle over $\widehat \mM_{\mH_{3m},R}$.

\subsection{Adiabatic limits and deformed sub-Dirac operators on
$\widehat \mM_{\mH_{3m},R}$}\label{s1.3}
Recall that we have assumed that $TM$ is oriented and spin. Thus $\mF\oplus\mF^\perp_{1} =\pi^*(TM)$ is spin. Without loss of generality, as in~\cite[p. 1062-1063]{Z17}, we can assume further that $F$ is oriented and ${\rm rk}(F^\perp)$ is divisible by $4$. Then $F^\perp$  is also oriented and $\dim \mM$ is even.

It is clear that $\mF\oplus \mF^\perp_{1},\, \mF^\perp_{2}$  over $\mM_{\mH_{3m},R}$ can be extended to
\begin{equation*}
  \big(\mH_{3m,R}\times [-1,2]\big)\allowbreak\cup \mM_{\mH_{3m},R}'
\end{equation*}
such that we have the orthogonal splitting\footnote{$\mF$ restricted to $(\mH_{3m,R}\times [-1,2])\cup \mM_{\mH_{3m},R}'$ needs no longer to be integrable.}
\begin{align*}\label{0.10}
T\widehat \mM_{\mH_{3m},R}= \left(\mF\oplus\mF^\perp_{1} \right)\oplus \mF^\perp_{2}\ \  {\rm on}\ \ \widehat \mM_{\mH_{3m},R}.
\end{align*}

Let $S_{\beta,\gamma}\left(\mF\oplus \mF^\perp_{1}\right)$ denote the spinor bundle over $\widehat \mM_{\mH_{3m},R}$ with respect to the metric  $g^{T\widehat\mM_{\mH_{3m},R}}|_{\mF \oplus\mF^\perp_{1}} $ (thus with respect to
$\beta^2g^{\mF}\oplus \frac{g^{\mF^\perp_{1}}}{\gamma^2}$ on $\mM_{\mH_{3m},R}$). Let $\Lambda^* \left(\mF_{2}^\perp\right )$ denote the exterior algebra bundle of $\mF_{2}^{\perp,*}$, with the   ${\mathbb{Z}}_2$-grading given by the natural even/odd parity.

Let
\begin{equation*}
  D _{\mF\oplus\mF_{1}^\perp,\beta,\gamma}:\Gamma\left (S_{\beta,\gamma} \left(\mF\oplus\mF_{1}^\perp\right)\widehat\otimes \Lambda^* \left(\mF_{2}^\perp \right)\right )
  \rightarrow \Gamma \left(S_{\beta,\gamma} \left(\mF\oplus\mF_{1}^\perp\right)\widehat \otimes \Lambda^*\left (\mF_{2}^\perp \right) \right)
\end{equation*}
be the sub-Dirac operator on $\widehat \mM_{\mH_{3m},R}$ constructed as in \cite[(2.16)]{Z17}.  
It is clear that one can define canonically the twisted sub-Dirac operator (twisted by $\mE_{3m,R}$) on $\widehat \mM_{\mH_{3m},R}$,
\begin{multline}\label{0.11}
D ^{\mE_{3m,R}}_{\mF\oplus\mF_{1}^\perp,\beta,\gamma}:\Gamma  \left(S_{\beta,\gamma}  \left(\mF\oplus\mF_{1}^\perp\right)\widehat\otimes
\Lambda^*\left (\mF_{2}^\perp\right )\widehat \otimes \mE_{3m,R}\right)
\\
\rightarrow
\Gamma \left (S_{\beta,\gamma}  \left(\mF\oplus\mF_{1}^\perp\right)\widehat\otimes
\Lambda^*\left(\mF_{2}^\perp\right )\widehat \otimes \mE_{3m,R}\right ).
\end{multline}

Let $\widetilde{f}:[0,1]\rightarrow [0,1]$ be a smooth function such that  $\widetilde{f}(t)= 0$ for $0\leq t\leq \frac{1}{4}$, while $\widetilde{f}(t) =1$ for $   \frac{1}{2}\leq t\leq 1$.  Let $h:[0,1]\rightarrow [0,1]$ be a smooth function such that $h(t)=1$ for $0\leq t\leq \frac{3}{4}$, while $h(t)=0$ for $\frac{7}{8}\leq t\leq 1$.

For any $p\in \mM_{\mH_{3m},R}$, we connect $p$ and $s(\pi(p))$ by the unique geodesic in $\mM_{ \pi(p)}$. Let $\sigma(p)\in \mF_{2}^\perp|_p$ denote the unit vector tangent to this geodesic. Then  
\begin{align}\label{0.13}
 \widetilde \sigma =\widetilde{ f}\left(\frac{\rho}{R}\right)\sigma
\end{align}
is a smooth section of $\mF_{2}^\perp|_{\mM_{\mH_{3m},R}}$. It extends to a smooth section of $\mF_{2}^\perp|_{\widehat\mM_{\mH_{3m},R}}$, which we still denote by $\widetilde\sigma$.  It is easy to see that we may and we will assume that $\widetilde\sigma$ is transversal to (and thus nowhere zero on) $\partial \widehat\mM_{\mH_{3m},R}$.

The Clifford action $\widehat c(\widetilde\sigma)$ (cf. \cite[(1.47)]{Z17}) now acts on
\begin{equation*}
  S_{\beta,\gamma}   \left(\mF\oplus\mF_{1}^\perp \right)\widehat\otimes
  \Lambda^* \left(\mF_{2}^\perp \right )\allowbreak\widehat{\otimes} \mE_{3m,R}
\end{equation*}
over $\widehat\mM_{\mH_{3m},R}$.

For $\varepsilon>0$, we introduce the following deformation of  $D ^{\mE_{3m,R}}_{\mF\oplus\mF_{1}^\perp,\beta,\gamma}$ on 
$\widehat\mM_{\mH_{3m},R}$ which put the deformations in \cite[(2.21)]{Z17} and \cite[(1.11)]{Z19} together,
\begin{align}\label{0.15}
 D ^{\mE_{3m,R}}_{\mF\oplus\mF_{1}^\perp,\beta,\gamma}
+\frac{\widehat c(\widetilde\sigma)}{\beta}
+\frac{\varepsilon \widehat{f}_{3m,R}^*V}{\beta}.
\end{align}

For this deformed sub-Dirac operator, we have the following analog of
\cite[Lemma 2.4]{Z17}.

\begin{lemma}\label{t0.4} {There exist $c_0>0$, $\varepsilon>0$, $m>0$ and $R>0$ such that when $\beta,\,\gamma>0$ are small enough (which may depend on $m$ and $R$), }

\emph{(i)} for any section $s\in \Gamma  \left(S_{\beta,\gamma}  \left (\mF\oplus\mF_{1}^\perp \right)\widehat\otimes
\Lambda^* \left(\mF_{2}^\perp\right )\widehat \otimes \mE_{3m,R} \right )$ supported in the interior of $\widehat\mM_{\mH_{3m},R}$, one has\footnote{The norms below   depend on $\beta$ and $\gamma$. In case of no confusion, we omit the subscripts for simplicity.}
\begin{align}\label{0.16}
 \bigg\|\bigg(D ^{\mE_{3m,R}}_{\mF\oplus\mF_{1}^\perp,\beta,\gamma}
+\frac{\widehat c(\widetilde\sigma)}{\beta}
+\frac{\varepsilon \widehat{f}_{3m,R}^*V}{\beta}\bigg)s\bigg\|\geq \frac{c_0 }{\beta}\|s\|;
\end{align}

\emph{(ii)} for any section $s\in \Gamma  \left(S_{\beta,\gamma}  \left (\mF\oplus\mF_{1}^\perp \right)\widehat\otimes
\Lambda^* \left(\mF_{2}^\perp \right)\widehat \otimes \mE_{3m,R} \right )$
supported in the interior of $\mM_{\mH_{3m},R}\setminus  \mM_{\mH_{3m},\frac{R}{2}}$, one has
\begin{align}\label{0.17}
 \bigg\|\bigg(h\left(\frac{\rho}{R}\right) D ^{\mE_{3m,R}}_{\mF\oplus\mF_{1}^\perp,\beta,\gamma}
 h\left(\frac{\rho}{R}\right) 
+\frac{\widehat c\left(\widetilde\sigma\right)}{\beta}
+\frac{\varepsilon \widehat{f}^*_{3m,R}V}{\beta}\bigg)s\bigg\|\geq \frac{c_0 }{\beta}\|s\|.
\end{align}
\end{lemma}

\begin{proof}

  Following \cite[Theorem 1.17]{GL83}, let $\phi:[0, \infty) \rightarrow [0,1]$ be a smooth function such that $\phi \equiv 1$ on $[0, 1]$, $\phi \equiv 0$ on $[2, \infty)$ and $\phi' \approx -1$ on $[1, 2]$. We define a smooth function {${\psi_m}:M_{H_{3m}}\to [0,1]$ by
  \begin{equation*}
    \psi_m(x) = \phi(d(x)/m),
  \end{equation*}
  where $m\in \mathbb{N}$. We extend $\psi_m$ to $(H_{3m}\times [-1,2])\cup M'_{H_{3m}}$ by setting
  \begin{equation*}
    \psi_m\big((H_{3m}\times [-1,2])\cup M'_{H_{3m}}\big)=0.
  \end{equation*}

Following~\cite[p. 115]{BL91}, let $\psi_{m,1},\, \psi_{m,2}: \widehat{M}_{H_{3m}}\rightarrow [0,1]$ be defined by
\begin{align}\label{0.20}
  \psi_{m,1} =\frac{\psi_{m}}{\big(\psi_{m}^2+(1-\psi_{m})^2\big)^{\frac{1}{2}}},\; \psi_{m,2} =\frac{1-\psi_{m}}{\big(\psi_{m}^2+(1-\psi_{m})^2\big)^{\frac{1}{2}}}.
\end{align}
Using the above defintion and (\ref{eq:fun-d}), for $i=1,2$, we have
\begin{equation}
  \label{eq:est-psi}
  |\nabla \psi_{m,i}|(x)\leq {C/m} \text{ for any }x\in \widehat{M}_{H_{3m}},
\end{equation}
where $C$ is a constant independent of $g^{TM}$.

We lift $\psi_m,\psi_{m,1},\psi_{m,2}$ to $\widehat \mM_{\mH_{3m},R}$ and denote them by $\varphi_m,\varphi_{m,1},\varphi_{m,2}$ respectively.} By definition, we have following properties about $\varphi_{m,1}$ and $\varphi_{m,2}$:
\begin{equation}
  \label{0.22}
  \begin{gathered}[]
    \varphi_{m,1}=1 \text{ if } x\in \pi^{-1}(B_m), \quad \varphi_{m,1}=0 \text{ if } x\in \widehat{\mM}_{\mH_{3m},R}\setminus \pi^{-1}(B_{2m});\\
    \varphi_{m,2}=0 \text{ if } x\in \pi^{-1}(B_m), \quad \varphi_{m,2}=1 \text{ if } x\in \widehat{\mM}_{\mH_{3m},R}\setminus \pi^{-1}(B_{2m}). \\
  \end{gathered}
\end{equation}

We first show the part (i) of Lemma \ref{t0.4}, {i.e.,}\ (\ref{0.16}).

For any $s\in \Gamma  \left(S_{\beta,\gamma}  \left (\mF\oplus\mF_{1}^\perp\right )\widehat\otimes
\Lambda^*\left (\mF_{2}^\perp \right)\widehat \otimes \mE_{3m,R} \right )$ supported in the interior of $\widehat\mM_{\mH_{3m},R}$, by (\ref{0.20}), one has
\begin{multline*}\label{0.21}
 \Big\|\Big(D ^{\mE_{3m,R}}_{\mF\oplus\mF_{1}^\perp,\beta,\gamma}+\frac{\widehat c(\widetilde\sigma)}{\beta} + \frac{\varepsilon \widehat{f}^*_{3m,R}V}{\beta}\Big)s\Big\|^2  = \Big\|\varphi_{m,1}\Big(D ^{\mE_{3m,R}}_{\mF\oplus\mF_{1}^\perp,\beta,\gamma}+\frac{\widehat c(\widetilde\sigma)}{\beta} + \frac{\varepsilon \widehat{f}^*_{3m,R}V}{\beta}\Big)s \Big\|^2 \\
+ \Big\| \varphi_{m,2}\Big(D ^{\mE_{3m,R}}_{\mF\oplus\mF_{1}^\perp,\beta,\gamma}+\frac{\widehat c(\widetilde\sigma)}{\beta} + \frac{\varepsilon \widehat{f}^*_{3m,R}V}{\beta}\Big)s\Big\|^2 ,
\end{multline*}
from which one gets, 
\begin{equation}\label{0.22a}
  \begin{aligned}[]
    & \sqrt{2}\Big\|\Big(D ^{\mE_{3m,R}}_{\mF\oplus\mF_{1}^\perp,\beta,\gamma}+\frac{\widehat c(\widetilde\sigma)}{\beta} + \frac{\varepsilon \widehat{f}^*_{3m,R}V}{\beta}\Big)s\Big\| \\
    & \qquad \geq
    \begin{multlined}[t][0.8\textwidth]
      \Big\|\varphi_{m,1}\Big(D ^{\mE_{3m,R}}_{\mF\oplus\mF_{1}^\perp,\beta,\gamma}+\frac{\widehat c(\widetilde\sigma)}{\beta} + \frac{\varepsilon \widehat{f}^*_{3m,R}V}{\beta}\Big)s\Big\|\\
      + \Big\| \varphi_{m,2}\Big(D ^{\mE_{3m,R}}_{\mF\oplus\mF_{1}^\perp,\beta,\gamma} + \frac{\widehat c(\widetilde\sigma)}{\beta} + \frac{\varepsilon \widehat{f}^*_{3m,R}V}{\beta}\Big)s\Big\|
    \end{multlined}\\
& \qquad \geq
    \begin{multlined}[t][0.8\textwidth]
      \Big\|\Big(D ^{\mE_{3m,R}}_{\mF\oplus\mF_{1}^\perp,\beta,\gamma}+\frac{\widehat c(\widetilde\sigma)}{\beta} + \frac{\varepsilon \widehat{f}^*_{3m,R}V}{\beta}\Big)(\varphi_{m,1}s)\Big\|\\
      + \Big\|\Big (D ^{\mE_{3m,R}}_{\mF\oplus\mF_{1}^\perp,\beta,\gamma}+\frac{\widehat c(\widetilde\sigma)}{\beta} + \frac{\varepsilon \widehat{f}^*_{3m,R}V}{\beta}\Big)\left(\varphi_{m,2}s\right)\Big\|\\
    -\Big\|c_{\beta,\gamma}\left({\rm d}\varphi_{m,1}\right)s\Big\|
    -\Big\|c_{\beta,\gamma}\left({\rm d}\varphi_{m,2}\right)s\Big\|,
    \end{multlined}
\end{aligned}
\end{equation}
where for each $i\in \{1,2\}$, we identify ${\dd \varphi_{m,i}}$ with the gradient of $\varphi_{m,i}$ and $c_{\beta,\gamma}(\cdot)$ means with respect to the metric (\ref{0.6}).

  We are going to estimate the r.h.s.\ of (\ref{0.22a}) term by term.
  We begin with a pointwise estimate of $c_{\beta,\gamma}(\dd \varphi_{m,i})s$, $i=1,2$. By (\ref{0.22}), we only need to do it on $\mM_{\mH_{3m},R}$.

  Here, as well as at several places in the following, we need to choose a local orthonormal frame for $T\mM_{\mH_{3m},R}$.
  Hence, we explain here our choice of this frame once and for all.
  Let $\mathrm{rk}(F)=\mathrm{rk}(\mathcal{F})=q$, $\mathrm{rk}(\mathcal{F}^{\perp}_1)= q_1$ and $\mathrm{rk}({\mathcal{F}^{\perp}_2}) = q_2$.
  Since on $\mM_{\mH_{3m},R}$, $g^{\mF} = \pi^* g^F$, for a local orthonormal basis $\{f_1,\dots,f_q\}$ of $(\mF, g^{\mF})$, we can choose it to be lifted from a local orthonormal basis of $(F, g^F)$.
  Moreover, we choose $h_1, \dots, h_{q_1}$ (resp.\ $e_1, \dots, e_{q_2}$) to be a local orthonormal basis of $(\mF^{\perp}_1, g^{\mF^{\perp}_1})$ (resp.\ $(\mF^{\perp}_2, g^{\mF^{\perp}_2})$).
  Then,
\begin{equation}
  \label{eq:loc-fr}
  \{f_1,\dots,f_q,h_1,\dots,h_{q_1},e_1,\dots,e_{q_2}\}
\end{equation}
is a local orthonormal frame for $T\mM_{\mH_{3m},R}$.

Back to the estimate of $c_{\beta,\gamma}(\dd \varphi_{m,i})s$, $i=1,2$.
Using the local frame (\ref{eq:loc-fr}), by (\ref{eq:est-psi}) and the fact that
\begin{equation*}
  \varphi_{m,i} = \psi_{m,i}\circ \pi,\quad i=1,2,
\end{equation*}
we have for any $1\leq k\leq q$ that 
\begin{equation*}
  |f_k(\varphi_{m,i}) |(x) = O\left({1\over {m}}\right), \text{ for any }x\in \mM_{\mH_{3m},R}
\end{equation*}
and for any $1\leq j\leq q_2$ that
\begin{align*}
e_j (\varphi_{m,i})=0.
\end{align*}
Therefore, by the properties of the Clifford action, for any $x\in\mM_{\mH_{3m},R}$, we have
\begin{multline}\label{2.19}
  |c_{\beta,\gamma}(\dd \varphi_{m,i})s|(x)
  \le \sum_{k=1}^q \Big(|\beta^{-1}f_k(\varphi_{m,i})|\cdot|c_{\beta,\gamma}(\beta^{-1}f_k)s|\Big)(x) \\
  + \sum_{l=1}^{q_1} \Big(|\gamma h_l(\varphi_{m,i})|\cdot|c_{\beta,\gamma}(\gamma h_l)s|\Big)(x)
  =\left(O\left({1\over {\beta m}}\right)+O_{m,R}(\gamma)\right) |s|(x),
\end{multline}
where the subscripts in $O_{m,R}(\cdot)$ mean that the estimating constant may depend on $m$ and $R$.

For the first two terms on the r.h.s.\ of (\ref{0.22a}),
by a direct computation, we have
\begin{multline}\label{721}
  \Big(D ^{\mE_{3m,R}}_{\mF\oplus\mF_{1}^\perp,\beta,\gamma}+\frac{\widehat c(\widetilde\sigma)}{\beta}
  +\frac{\varepsilon \widehat{f}^*_{3m,R}V}{\beta}\Big)^2=\Big(D ^{\mE_{3m,R}}_{\mF\oplus\mF_{1}^\perp,\beta,\gamma}+\frac{\widehat c(\widetilde\sigma)}{\beta}   \Big)^2 \\
  +\Big[D ^{\mE_{3m,R}}_{\mF\oplus\mF_{1}^\perp,\beta,\gamma},{{\varepsilon \widehat{f}^*_{3m,R}V}\over{\beta}}\Big]
  +\frac{\varepsilon^2 (\widehat{f}^*_{3m,R}V)^2}{\beta^2}.
\end{multline} 
Of the three terms on the r.h.s.\ of the above
  equality, we can control the last term relatively easily.

  Now we will deal with the
  second term on the r.h.s.\ of the above equality. By~\cite[(2.17)]{Z17}, using the local frame (\ref{eq:loc-fr}), one has
\begin{multline}
  \label{eq:d-v}
  \Big[D ^{\mE_{3m,R}}_{\mF\oplus\mF_{1}^\perp,\beta,\gamma},{{\varepsilon \widehat{f}^*_{3m,R}V}\over{\beta}}\Big] = \sum^{q}_{i=1} \beta^{-1}c_{\beta,\gamma}(\beta^{-1}f_i)\Big[\nabla^{\mE_{3m,R}}_{f_i}, {{\varepsilon \widehat{f}^*_{3m,R}V}\over{\beta}}\Big] \\
  + \sum^{q_1}_{s=1} \gamma c_{\beta,\gamma}(\gamma h_s) \Big[\nabla^{\mE_{3m,R}}_{h_s}, {{\varepsilon \widehat{f}^*_{3m,R}V}\over{\beta}}\Big]
  + \sum^{q_2}_{j=1} c_{\beta,\gamma}(e_j)\Big[\nabla^{\mE_{3m,R}}_{e_j}, {{\varepsilon \widehat{f}^*_{3m,R}V}\over{\beta}}\Big].
\end{multline}

Since $\nabla^{\mE_{3m,R}}$ (resp.\ $\widehat{f}^*_{3m,R}V$) is a  pull-back connection (resp.\ bundle endomorphism) via $\pi$, we have
\begin{equation}
  \label{eq:n-v}
  \begin{gathered}[]
\left [\nabla^{\mE_{3m,R}}_{e_j}, {{\widehat{f}^*_{3m,R}V}}\right] = 0, \\
   \left [\nabla^{\mE_{3m,R}}_{f_i}, {{\widehat{f}^*_{3m,R}V}}\right]  = O(1),\\
   \left [\nabla^{\mE_{3m,R}}_{h_s}, {{\widehat{f}^*_{3m,R}V}}\right] = O_R(1).
  \end{gathered}
\end{equation}
Putting (\ref{eq:d-v}) and (\ref{eq:n-v}) together, one has
\begin{align}\label{2.26a}
\Big[D ^{\mE_{3m,R}}_{\mF\oplus\mF_{1}^\perp,\beta,\gamma},{{\varepsilon \widehat{f}^*_{3m,R}V}\over{\beta}}\Big]=O\left({\varepsilon\over \beta^2}\right)+O_{R}\left({\varepsilon\gamma\over \beta}\right)\ {\rm on}\ \pi^{-1}({\rm Supp}(\dd f)).
\end{align}
Meanwhile, since {$f^*V$} is a constant endomorphism outside the support of $\dd f$, we know
\begin{align}\label{2.27a}
\Big[D ^{\mE_{3m,R}}_{\mF\oplus\mF_{1}^\perp,\beta,\gamma},{{\varepsilon \widehat{f}^*_{3m,R}V}\over{\beta}}\Big]= 0\ {\rm on}\ \widehat{\mM}_{\mH_{3m},R}\setminus {\pi^{-1}({\rm Supp}(\dd f))}.
\end{align}

The first term on the r.h.s.\ of (\ref{721}) is nonnegative. But for our purpose, such an estimate is not enough. We need analyze it more precisely, especially on $\mM_{\mH_{3m},R}$. In fact, using the local frame (\ref{eq:loc-fr}), by \cite[(2.24) and (2.28)]{Z17} (see also \cite[(1.13)]{SZ}), on $\mM_{\mH_{3m},R}$, we have
\begin{multline}\label{2.24}
 \Big (D ^{\mE_{3m,R}}_{\mF\oplus\mF_{1}^\perp,\beta,\gamma}+\frac{\widehat c(\widetilde\sigma)}{\beta}\Big)^2=\Big(D ^{\mE_{3m,R}}_{\mF\oplus\mF_{1}^\perp,\beta,\gamma}\Big)^2 +\Big [D ^{\mE_{3m,R}}_{\mF\oplus\mF_{1}^\perp,\beta,\gamma},\frac{\widehat c(\widetilde\sigma)}{\beta}\Big] + {{|\widetilde{\sigma}|^2}\over\beta^2}\\
  =-\Delta^{\mE_{3m,R},\beta,\gamma}+{{k^\mF}\over{4\beta^2}}  + \sum_{i,j=1}^{q}{1\over{2\beta^2}} R^{\mE_{3m,R}}(f_i,f_j)c_{\beta,\gamma}(\beta^{-1}f_i)c_{\beta,\gamma}(\beta^{-1}f_j)\\
  + \Big[D ^{\mE_{3m,R}}_{\mF\oplus\mF_{1}^\perp,\beta,\gamma},\frac{\widehat c(\widetilde\sigma)}{\beta}\Big]
+{{|\widetilde{\sigma}|^2}\over\beta^2}+O_{m,R}\left({1\over \beta}+{\gamma^2\over \beta^2}\right),
\end{multline}
where $-\Delta^{\mE_{3m,R},\beta,\gamma}\geq 0$ is the corresponding Bochner Laplacian, $k^\mF=\pi^* (k^F)$ and
\begin{equation*}
  R^{\mE_{3m,R}}=(\nabla^{\mE_{3m,R,+}})^2+(\nabla^{\mE_{3m,R,-}})^2.
\end{equation*}

The benefits of (\ref{2.24}) is that each term on the r.h.s.\ of it can be controlled in a certain way. More concretely,
since $f$ is area decreasing along $F$, by \cite[(2.6)]{S} which goes back to \cite{LL}, we have
\begin{multline}\label{2.25}
  \Big({1\over{2\beta^2}}\sum_{i,j=1}^{q}R^{\mE_{3m,R}}(f_i,f_j)c_{\beta,\gamma}(\beta^{-1}f_i)c_{\beta,\gamma}(\beta^{-1}f_j)s,s\Big)_{\pi^{-1}({\rm Supp}(\dd f))}\\
  \geq {-{{q(q-1)}\over{4\beta^2}}}\|s\|^2_{\pi^{-1}({\rm Supp}(\dd f))},
\end{multline}
where $(\cdot\ ,\ \cdot )_{\pi^{-1}({\rm Supp}(\dd f))}$ and $\|\cdot \|_{\pi^{-1}({\rm Supp}(\dd f))}$ mean the integration on ${\pi^{-1}({\rm Supp}(\dd f))}$.

By (\ref{1.2}), there exists $\kappa>0$ such that 
\begin{align}\label{g2.24}
k^\mF-q(q-1)\geq \kappa\ {\rm on}\ \pi^{-1}({\rm Supp}(\dd f)).
\end{align}
{And }by \cite[Lemma 2.1]{Z17}, on $\mM_{\mH_{3m},R}\setminus s(M_{H_{3m}})$, we have
\begin{align}\label{2.27}
\left[D ^{\mE_{3m,R}}_{\mF\oplus\mF_{1}^\perp,\beta,\gamma},\frac{\widehat c(\widetilde\sigma)}{\beta}\right]=O_m\left({1\over \beta^2 R}\right)+O_{m,R}\left(1\over \beta\right).
\end{align}

{Now we can estimate the first term on r.h.s.\ of (\ref{0.22a}). As a first step,}
using (\ref{721}) and (\ref{2.24}), since the terms involving the Bochner Laplacian and $|\widetilde{\sigma}|$ are nonnegative, we have
  \begin{multline}\label{2.35j}
   \Big \|\Big(D ^{\mE_{3m,R}}_{\mF\oplus\mF_{1}^\perp,\beta,\gamma}+\frac{\widehat c(\widetilde\sigma)}{\beta}
    +\frac{\varepsilon \widehat{f}^*_{3m,R}V}{\beta}\Big)(\varphi_{m,1}s)\Big\|^2 \ge \Big({{k^\mF}\over{4\beta^2}}\varphi_{m,1}s, \varphi_{m,1}s\Big) \\ +\Big(\sum_{i,j=1}^{q}{1\over{2\beta^2}}R^{\mE_{3m,R}}(f_i,f_j)c_{\beta,\gamma}(\beta^{-1}f_i)c_{\beta,\gamma}(\beta^{-1}f_j)\varphi_{m,1}s, \varphi_{m,1}s\Big) \\
    +\Big(\Big[D ^{\mE_{3m,R}}_{\mF\oplus\mF_{1}^\perp,\beta,\gamma},\frac{\widehat c(\widetilde\sigma)}{\beta}\Big]\varphi_{m,1}s, \varphi_{m,1}s\Big)
    +\Big(\Big[D ^{\mE_{3m,R}}_{\mF\oplus\mF_{1}^\perp,\beta,\gamma},{{\varepsilon \widehat{f}^*_{3m,R}V}\over{\beta}}\Big]\varphi_{m,1}s,\varphi_{m,1}s\Big) \\
+{\Big( \frac{\varepsilon^2 (\widehat{f}^*_{3m,R}V)^2}{\beta^2}\varphi_{m,1}s,\varphi_{m,1}s\Big)}+\Big(O_{m,R}\left({1\over \beta}+{\gamma^2\over \beta^2}\right)\varphi_{m,1}s, \varphi_{m,1}s\Big) = \mathrm{I} + \mathrm{II},
  \end{multline}
  where we need to explain the meaning of symbols $\mathrm{I},\mathrm{II}$ appearing on the rightmost.
  Note that every term on the r.h.s.\ of the above inequality can be written as the sum of two parts: the integral on $\pi^{-1}(\supp(\dd f))$ and the integral on $\pi^{-1}(B_{2m}\setminus \supp(\dd f))$. We denote the sum of all the integrals on $\pi^{-1}(\supp(\dd f))$ (resp.\ $\pi^{-1}(B_{2m}\setminus \supp(\dd f))$) by the symbol $\mathrm{I}$ (resp.\ $\mathrm{II}$).

By (\ref{eq:df-b}) and (\ref{0.22}), we have $\varphi_{m,1}s = s$ on $\pi^{-1}(\supp(\dd f))$.
Therefore, by (\ref{2.26a}), (\ref{2.25}), (\ref{g2.24}), (\ref{2.27}) and proceeding as in \cite[p. 1058-1059]{Z17}, one has
\begin{equation}\label{2.31a}
\begin{aligned}[]
    \mathrm{I} &\geq
    \begin{multlined}[t]
      {\kappa\over 4\beta^2}\|s\|^2_{\pi^{-1}({\rm Supp}(\dd f))}+O\left(\varepsilon\over \beta^2\right)\|s\|^2_{\pi^{-1}({\rm Supp}(\dd f))} +O_{R}\left({\varepsilon\gamma\over \beta}\right)\|s\|^2_{\pi^{-1}({\rm Supp}(\dd f))} \\
      +O\left(1\over \beta^2 R\right)\|s\|^2_{\pi^{-1}({\rm Supp}(\dd f))}+O_{m,R}\left({1\over \beta}+{\gamma^2\over\beta^2}\right)\|s\|^2_{\pi^{-1}({\rm Supp}(\dd f))}
    \end{multlined}\\
    &=
    \begin{multlined}[t]
    {\kappa\over 8\beta^2}\|s\|^2_{\pi^{-1}({\rm Supp}(\dd f))} +\left( {\kappa\over 8\beta^2}\|s\|^2_{\pi^{-1}({\rm Supp}(\dd f))}
    +O\left(\varepsilon\over \beta^2\right)\|s\|^2_{\pi^{-1}({\rm Supp}(\dd f))}\right)\\+O_{R}\left({\varepsilon\gamma\over \beta}\right)\|s\|^2_{\pi^{-1}({\rm Supp}(\dd f))}
    +O\left(1\over \beta^2 R\right)\|s\|^2_{\pi^{-1}({\rm Supp}(\dd f))}\\+O_{m,R}\left({1\over \beta}+{\gamma^2\over\beta^2}\right)\|s\|^2_{\pi^{-1}({\rm Supp}(\dd f))}    
    \end{multlined}\\
    &\geq
    \begin{multlined}[t]
    {\kappa\over 8\beta^2}\|s\|^2_{\pi^{-1}({\rm Supp}(\dd f))} +O_{R}\left({\varepsilon\gamma\over \beta}\right)\|s\|^2_{\pi^{-1}({\rm Supp}(\dd f))}\\
    +O\left(1\over \beta^2 R\right)\|s\|^2_{\pi^{-1}({\rm Supp}(\dd f))}+O_{m,R}\left({1\over \beta}+{\gamma^2\over\beta^2}\right)\|s\|^2_{\pi^{-1}({\rm Supp}(\dd f))},
    \end{multlined}
  \end{aligned}
\end{equation}
where for the last inequality, we have chosen small enough $\varepsilon>0$ such that
\begin{equation*}
  {\kappa\over 8\beta^2}\|s\|^2_{\pi^{-1}({\rm Supp}(\dd f))}
  +O\left(\varepsilon\over \beta^2\right)\|s\|^2_{\pi^{-1}({\rm Supp}(\dd f))}\geq 0.
\end{equation*}

By (\ref{2.11a}), on $\widehat \mM_{\mH_{3m},R}\setminus \pi^{-1}\left({\rm Supp}(\dd f)\right)$, one has
\begin{align}\label{2.9}
{(\widehat{f}^*_{3m,R}V)}^2\geq \delta.
\end{align}

Recall that on $M$, we have assumed that $k^{F}$ is nonnegative, i.e.,~(\ref{eq:kf-lb}). As a result, $\inf (k^\mF)\geq 0$ holds on $\pi^{-1}(B_{2m}\setminus {\rm Supp}({\rm d}f))$.
Moreover, on $\pi^{-1}(B_{2m}\setminus {\rm Supp}({\rm d}f))$, we have $R^{\mE_{3m,R}}=0$. Therefore, from (\ref{2.27a}), (\ref{2.27}), (\ref{2.9}) and proceeding as in \cite[p. 1058-1059]{Z17},  we obtain
\begin{multline}\label{2.32a}
\mathrm{II} \geq {\delta\varepsilon^2\over \beta^2}\|\varphi_{m,1}s\|^2_{\pi^{-1}(B_{2m}\setminus{\rm Supp}(\dd f))}\\
  +O_{m,R}\left(1\over \beta\right)\|\varphi_{m,1}s\|^2_{\pi^{-1}(B_{2m}\setminus{\rm Supp}(\dd f))}\\
  +O_{m,R}\left(\gamma^2\over \beta^2\right)\|\varphi_{m,1}s\|^2_{\pi^{-1}(B_{2m}\setminus{\rm Supp}(\dd f))}\\
  +O_m\left(1\over\beta^2 R\right)\|\varphi_{m,1}s\|^2_{\pi^{-1}(B_{2m}\setminus{\rm Supp}(\dd f))}.
\end{multline}

For the second term on the r.h.s.\ of (\ref{0.22a}), by (\ref{eq:df-b}), (\ref{0.22}), (\ref{721}), (\ref{2.27a}) and (\ref{2.9}), one has
\begin{multline}\label{2.33}
\Big\|\Big(D ^{\mE_{3m,R}}_{\mF\oplus\mF_{1}^\perp,\beta,\gamma}+\frac{\widehat c(\widetilde\sigma)}{\beta}
    +\frac{\varepsilon \widehat{f}^*_{3m,R}V}{\beta}\Big)(\varphi_{m,2}s)\Big\|^2\\
    =\Big \|\Big(D ^{\mE_{3m,R}}_{\mF\oplus\mF_{1}^\perp,\beta,\gamma}+\frac{\widehat c(\widetilde\sigma)}{\beta}\Big)(\varphi_{m,2}s)\Big\|^2 +
    \Big\|\frac{\varepsilon \widehat{f}^*_{3m,R}V}{\beta}(\varphi_{m,2}s)\Big\|^2
    \geq {{\delta\varepsilon^2}\over{\beta^2}}\left\|\varphi_{m,2}s\right\|^2.
\end{multline}

From (\ref{2.35j}), (\ref{2.31a}), (\ref{2.32a}) and (\ref{2.33}), one has
\begin{multline}\label{2.34}
  \sum_{j=1}^2\Big\|\Big(D ^{\mE_{3m,R}}_{\mF\oplus\mF_{1}^\perp,\beta,\gamma}+\frac{\widehat c(\widetilde\sigma)}{\beta}
  +\frac{\varepsilon \widehat{f}^*_{3m,R}V}{\beta}\Big)(\varphi_{m,j}s)\Big\|^2
  \geq \min\left\{{\kappa\over 8},\delta\varepsilon^2\right\}{\|s\|^2\over \beta^2}\\
  +O_{R}\left(\varepsilon\gamma\over \beta\right)\|s\|^2_{\pi^{-1}({\rm Supp}(\dd f))}+O_{m,R}\left(\gamma^2\over \beta^2\right)\|\varphi_{m,1}s\|^2\\
  +O_{m,R}\left(1\over \beta\right)\|\varphi_{m,1}s\|^2+O_{m}\left({1\over \beta^2 R}\right)\|\varphi_{m,1}s\|^2.
\end{multline}

From (\ref{0.22a}), (\ref{2.19}) and (\ref{2.34}), by taking $m$ sufficiently
large and then taking $R$ sufficiently large, one finds that there exist $c_0>0$, $\varepsilon>0$, $m>0$ and $R>0$ such that when $\beta>0, \gamma>0$ are small enough (\ref{0.16}) holds, i.e.,\ part (i) of the lemma.

{The strategy to prove part (ii) of the lemma is similar to that of part (i).} For any smooth section $s$ in question, one has as in (\ref{0.22a}) that
\begin{multline}\label{1.26e}
 \sqrt{2}\Big\|\Big( h\left(\frac{\rho}{R}\right)  D ^{\mE_{3m,R}}_{\mF\oplus\mF_{1}^\perp,\beta,\gamma} h\left(\frac{\rho}{R}\right)  
+\frac{\widehat c(\widetilde\sigma)}{\beta}
+\frac{\varepsilon\widehat{f}^*_{3m,R}V}{\beta}\Big)s\Big\|
\\
\geq \Big\|\Big(h\left(\frac{\rho}{R}\right)  D ^{\mE_{3m,R}}_{\mF\oplus\mF_{1}^\perp,\beta,\gamma}h\left(\frac{\rho}{R}\right)  
+\frac{\widehat c(\widetilde\sigma)}{\beta}
+\frac{\varepsilon\widehat{f}^*_{3m,R}V}{\beta}\Big)(\varphi_{m,1}s)\Big\|
\\
\qquad+\Big\| \Big(h\left(\frac{\rho}{R}\right)  D ^{\mE_{3m,R}}_{\mF\oplus\mF_{1}^\perp,\beta,\gamma}h\left(\frac{\rho}{R}\right)  
+\frac{\widehat c(\widetilde\sigma)}{\beta}
+\frac{\varepsilon\widehat{f}^*_{3m,R}V}{\beta}\Big)(\varphi_{m,2}s)\Big\| \\
-\|c_{\beta,\gamma}({\rm d}\varphi_{m,1})s\|
-\|c_{\beta,\gamma}({\rm d}\varphi_{m,2})s\|.
\end{multline}

By a direct calculation (comparing with \cite[(2.29)]{Z17}),
\begin{equation}
  \label{0.25}
  \begin{aligned}
    &\quad \Big(h\left(\frac{\rho}{R}\right) D ^{\mE_{3m,R}}_{\mF\oplus\mF_{1}^\perp,\beta,\gamma}
    h\left(\frac{\rho}{R}\right) 
    +\frac{\widehat c\left(\widetilde\sigma\right)}{\beta}
    +\frac{\varepsilon\widehat{f}^*_{3m,R}V}{\beta}\Big)^2\\
    =&
    \begin{multlined}[t][0.9\textwidth]
      \Big(h\left(\frac{\rho}{R}\right) D ^{\mE_{3m,R}}_{\mF\oplus\mF_{1}^\perp,\beta,\gamma}
      h\left(\frac{\rho}{R}\right) + \frac{\widehat c\left(\widetilde\sigma\right)}{\beta}\Big)^2\\
      +h\left(\frac{\rho}{R}\right) ^2\Big[D ^{\mE_{3m,R}}_{\mF\oplus\mF_{1}^\perp,\beta,\gamma}
      ,\frac{\varepsilon\widehat{f}^*_{3m,R}V}{\beta}\Big]
      +\frac{\varepsilon^2(\widehat{f}^*_{3m,R}V)^2}{\beta^2}
    \end{multlined}\\
    =&
    \begin{multlined}[t][0.9\textwidth]
      \Big(h\left(\frac{\rho}{R}\right) D ^{\mE_{3m,R}}_{\mF\oplus\mF_{1}^\perp,\beta,\gamma}
      h\left(\frac{\rho}{R}\right)\Big)^2
      +\frac{h(\frac{\rho}{R})^2}{\beta}\left[D ^{\mE_{3m,R}}_{\mF\oplus\mF_{1}^\perp,\beta,\gamma},\widehat c(\widetilde\sigma)\right]+\frac{|\widetilde\sigma|^2}{\beta^2}\\
      +\frac{h\left(\frac{\rho}{R}\right)^2}{\beta}\left [ D ^{\mE_{3m,R}}_{\mF\oplus\mF_{1}^\perp,\beta,\gamma}
      , {\varepsilon\widehat{f}^*_{3m,R}V}\right]
      +\frac{\varepsilon^2(\widehat{f}^*_{3m,R}V)^2}{\beta^2}.
    \end{multlined}
  \end{aligned}
\end{equation}

We estimate the second term on the r.h.s.\ of (\ref{1.26e}) first. By (\ref{eq:df-b}), (\ref{0.22}), (\ref{2.27a}), (\ref{2.9}) and the first equality in (\ref{0.25}), one has
\begin{multline}\label{2.35a}
  \Big\|\Big(h\left(\frac{\rho}{R}\right) D ^{\mE_{3m,R}}_{\mF\oplus\mF_{1}^\perp,\beta,\gamma}
  h\left(\frac{\rho}{R}\right) 
  +\frac{\widehat c\left(\widetilde\sigma\right)}{\beta}
  +\frac{\varepsilon \widehat{f}^*_{3m,R}V}{\beta}\Big)(\varphi_{m,2}s)\Big\|^2\\
= \Big\|\Big(h\left(\frac{\rho}{R}\right) D ^{\mE_{3m,R}}_{\mF\oplus\mF_{1}^\perp,\beta,\gamma} h\left(\frac{\rho}{R}\right) + \frac{\widehat c\left(\widetilde\sigma\right)}{\beta}\Big)(\varphi_{m,2})s\Big\|^2 + \Big\|\frac{\varepsilon \widehat{f}^*_{3m,R}V}{\beta}(\varphi_{m,2}s)\Big\|^2\\
\geq {\varepsilon^2\delta\over \beta^2}\left\|\varphi_{m,2}s\right\|^2.
\end{multline}

To estimate the first term on the r.h.s.\ of (\ref{1.26e}), we use the following fact. By the definition of $\widetilde{\sigma}$, since now
\begin{equation*}
  \supp{(s)} \subseteq \mM_{\mH_{3m},R}\setminus  \mM_{\mH_{3m},\frac{R}{2}},
\end{equation*}
we have by (\ref{0.13})
\begin{align}
  \label{eq:est-sigma}
  \left (\frac{|\widetilde\sigma|^2}{\beta^2} \varphi_{m,1}s, \varphi_{m,1}s\right) \ge \frac{1}{\beta^2}\| \varphi_{m,1} s\|^2.
\end{align}

From (\ref{2.26a}), (\ref{2.27a}), (\ref{2.27}), {(\ref{eq:est-sigma})} and the second equality in (\ref{0.25}), one gets
 \begin{equation}\label{2.36a}
  \begin{aligned}
    & \quad\Big\|\Big(h\left(\frac{\rho}{R}\right) D ^{\mE_{3m,R}}_{\mF\oplus\mF_{1}^\perp,\beta,\gamma} h\left(\frac{\rho}{R}\right) +\frac{\widehat c\left(\widetilde\sigma\right)}{\beta} + \frac{\varepsilon \widehat{f}^*_{3m,R}V}{\beta}\Big)(\varphi_{m,1}s)\Big\|^2\\
    \ge&
    \begin{multlined}[t][0.9\textwidth]
      \left(\frac{|\widetilde\sigma|^2}{\beta^2} \varphi_{m,1}s, \varphi_{m,1}s\right) + \left(\frac{h(\frac{\rho}{R})^2}{\beta}\left[D ^{\mE_{3m,R}}_{\mF\oplus\mF_{1}^\perp,\beta,\gamma},{\widehat c(\widetilde\sigma)}\right]\varphi_{m,1}s,\varphi_{m,1}s\right) \\
        + \left(\frac{h(\frac{\rho}{R})^2}{\beta}\left[ D ^{\mE_{3m,R}}_{\mF\oplus\mF_{1}^\perp,\beta,\gamma}, {\varepsilon\widehat{f}^*_{3m,R}V}\right] \varphi_{m,1}s, \varphi_{m,1}s\right)
    \end{multlined}\\
    \geq&
    \begin{multlined}[t][0.9\textwidth]
      {1\over \beta^2}\|\varphi_{m,1}s\|^2+\Big(O_m\left(1\over \beta^2 R\right)+O_{m,R}\left(1\over \beta\right)\Big)\|\varphi_{m,1}s\|^2\\
      +\Big(O\left(\varepsilon\over\beta^2\right)+O_{R}\left(\varepsilon\gamma\over \beta\right)\Big)\|\varphi_{m,1}s\|^2_{\pi^{-1}({\rm Supp}(\dd f))}
    \end{multlined}\\
    =&
    \begin{multlined}[t][0.9\textwidth]
      {1\over 2\beta^2}\|\varphi_{m,1}s\|^2+\Big(O_m\left(1\over \beta^2 R\right)+O_{m,R}\left(1\over \beta\right)\Big)\|\varphi_{m,1}s\|^2\\
      +\Big( {1\over 2\beta^2}\|\varphi_{m,1}s\|^2+O\left(\varepsilon\over\beta^2\right)\|\varphi_{m,1}s\|^2_{\pi^{-1}({\rm Supp}(\dd f))}\Big)\\+O_R\left(\varepsilon\gamma\over \beta\right)\|\varphi_{m,1}s\|^2_{\pi^{-1}({\rm Supp}(\dd f))}
    \end{multlined}\\   
   \geq&
   \begin{multlined}[t][0.9\textwidth]
     {1\over 2\beta^2}\|\varphi_{m,1}s\|^2+ \Big(O_m\left(1\over \beta^2 R\right)+O_{m,R}\left({1\over \beta}\right)\Big)\|\varphi_{m,1}s\|^2 \\
     + O_R\left(\varepsilon\gamma\over \beta\right)\|\varphi_{m,1}s\|^2_{\pi^{-1}({\rm Supp}(\dd f))},
    \end{multlined}
  \end{aligned}
\end{equation} 
where for the last inequality, we have chosen small enough $\varepsilon>0$ such that
 \begin{equation*}
  {1\over 2\beta^2}\|\varphi_{m,1}s\|^2+O\left(\varepsilon\over\beta^2\right)\|\varphi_{m,1}s\|^2_{\pi^{-1}({\rm Supp}(\dd f))}\geq 0.
 \end{equation*}

From (\ref{2.35a}) and (\ref{2.36a}), one gets
\begin{multline}\label{2.39}
\sum_{j=1}^2\Big\|\Big(h\left(\frac{\rho}{R}\right) D ^{\mE_{3m,R}}_{\mF\oplus\mF_{1}^\perp,\beta,\gamma}
 h\left(\frac{\rho}{R}\right) 
+\frac{\widehat c\left(\widetilde\sigma\right)}{\beta}
+\frac{\varepsilon \widehat{f}^*_{3m,R}V}{\beta}\Big)(\varphi_{m,j}s)\Big\|^2\\
\geq \min\{1,\varepsilon^2\delta\}{\|s\|^2\over 2\beta^2}+\Big(O_m\left(1\over \beta^2 R\right)+O_{m,R}\left({1\over \beta}\right)\Big)\|\varphi_{m,1}s\|^2\\
+O_R\left(\varepsilon\gamma\over \beta\right)\|\varphi_{m,1}s\|^2_{\pi^{-1}({\rm Supp}(\dd f))}.
\end{multline}

From (\ref{2.19}), {(\ref{1.26e})} and (\ref{2.39}), by taking $m$ sufficiently large and then taking $R$ sufficiently large, one finds that there exist $c_0>0$, $\varepsilon>0$, $m>0$ and $R>0$ such that when $\beta>0, \gamma>0$ are small enough (\ref{0.17}) holds, {i.e.,\ part (ii) of the lemma}.

\end{proof}

\subsection{Elliptic operators on
$\widetilde \mN_{ 3m,R}$}\label{s1.4}
Let $Q$ be a Hermitian vector bundle over $\widehat \mM_{\mH_{3m},R}$ such that $\left(S_{\beta,\gamma} \left (\mF\oplus\mF_{1}^\perp\right)\widehat\otimes
\Lambda^*\left(\mF_{2}^\perp \right)\widehat \otimes \mE_{m,R}\right)_{-}\oplus Q$ is a trivial vector bundle over $\widehat \mM_{\mH_{3m},R}$. Then, under the identification
\begin{equation*}
  \widehat c(\widetilde{\sigma})+\widehat f_{3m,R}^*v+{\rm Id}_Q,
\end{equation*}
$\left(S_{\beta,\gamma}  \left(\mF\oplus\mF_{1}^\perp\right)\widehat\otimes
\Lambda^*\left(\mF_{2}^\perp \right)\widehat \otimes \mE_{m,R}\right)_{+}\oplus Q$ is a trivial vector bundle near $\partial \widehat \mM_{\mH_{3m},R}$.

By obviously extending the above trivial vector bundles to $\widetilde \mN _{3m,R}$, we get a ${\mathbb{Z}}_2$-graded Hermitian vector bundle $\xi=\xi_+\oplus \xi_-$ over $\widetilde  \mN_{3m,R}$ and an odd self-adjoint endomorphism $W=w+w^*\in \Gamma({\rm End}(\xi))$ (with $w:\Gamma(\xi_+)\to \Gamma(\xi_-)$, $w^*$ being the adjoint of $w$) such that 
\begin{align*}\xi_\pm=\left(S_{\beta,\gamma}  \left(\mF\oplus\mF_{1}^\perp\right)\widehat\otimes
\Lambda^*\left(\mF_{2}^\perp\right )\widehat \otimes \mE_{m,R}\right)_\pm\oplus Q
\end{align*}
over $\widehat \mM_{\mH_{3m},R}$, $W$ is invertible on $\mN _{3m,R}$ and 
\begin{align}\label{0.27}
  W=  \widehat c\left(\widetilde\sigma\right)+\widehat{f}^*_{3m,R}V+
  \begin{pmatrix}0 & {\rm Id}_Q\\
    {\rm Id}_Q & 0\\
  \end{pmatrix}
\end{align}
on $\widehat \mM_{\mH_{3m},R}$, which is invertible on $\widehat\mM_{\mH_{3m},R}\setminus  \mM_{\mH_{3m},R/2}$.

Recall that $h({\rho}/{R})$ vanishes near $\mM_{\mH_{3m},R}\cap \partial  \mM_{R}$. We extend it to a function on $\widetilde \mN_{ 3m,R}$ which equals zero on $\mN _{3m,R}$ and an open neighborhood of $\partial \widehat \mM_{\mH_{3m},R}$ in $\widetilde \mN_{ 3m,R}$, and we denote the resulting function on $\widetilde \mN_{3m,R}$ by $\widetilde h_R$.

Let $\pi_{\widetilde \mN_{3m, R}}: T\widetilde \mN_{3m,R}\to \widetilde \mN_{3m, R}$ be the projection of the tangent bundle of $\widetilde \mN_{3 m,R}$. Let
\begin{equation*}
  \gamma^{\widetilde \mN_{3 m,R}}\in {\rm Hom}\big(\pi^*_{\widetilde \mN_{3m, R}}\xi_+,\pi^*_{\widetilde \mN_{ 3m,R}}\xi_-\big)
\end{equation*}
be the symbol defined for $p\in \widetilde \mN_{ 3m,R}$ and $u\in T_p \widetilde \mN_{ 3m,R}$ by 
\begin{align}\label{0.28}
  \gamma^{\widetilde \mN_{ 3m,R}}(p,u)=\pi^*_{\widetilde \mN_{3m, R}}\big(\sqrt{-1} \widetilde h ^2_R c_{\beta,\gamma}(u)+w(p)\big).
\end{align}
By (\ref{0.27}) and (\ref{0.28}), $\gamma^{\widetilde\mN_{3 m,R}}$ is singular only if $u=0$ and $p\in\mM_{\mH_{3m},R/2}$. Thus $\gamma^{\widetilde \mN_{ 3m,R}}$ is an elliptic symbol.

On the other hand, it is clear that $\widetilde h_R  D ^{\mE_{3m,R}}_{\mF\oplus\mF_{1}^\perp,\beta,\gamma}\widetilde h_R$ is well defined on $\widetilde \mN_{ 3m,R}$ if we define it to equal to zero on $\widetilde \mN_{ 3m,R}\setminus \widehat \mM_{\mH_{3m},R}$.

Let $A: L^2 (\xi)\to L^2 (\xi)$ be a second order positive elliptic differential operator on $\widetilde \mN_{ m,R}$ preserving the ${\mathbb{Z}}_2$-grading of $\xi=\xi_+\oplus \xi_-$, such that its symbol equals $|\eta|^2$ at $\eta\in T\widetilde \mN_{ 3m,R}$.\footnote{To be more precise, here $A$ also depends on the defining metric. We omit the corresponding subscript/superscript only for convenience.}
As in \cite[(2.33)]{Z17}, let
\begin{equation*}
  P^{\mE_{3m,R}}_{ R,\beta,\gamma}:L^2 (\xi)\to L^2 (\xi)
\end{equation*}
be the zeroth order pseudodifferential operator on $\widetilde \mN_{3m, R}$ defined by
\begin{align}\label{0.29}
P^{\mE_{3m,R}}_{R,\beta,\gamma}=A^{-\frac{1}{4}}\widetilde h_R D ^{\mE_{3m,R}}_{\mF\oplus\mF_{1}^\perp,\beta,\gamma}\widetilde h_R A^{-\frac{1}{4}}+\frac{W}{\beta}.
\end{align}
Let
\begin{equation*}
  P^{\mE_{3m,R}}_{R,\beta,\gamma,+}:L^2(\xi_+)\to L^2(\xi_-)
\end{equation*}
be the obvious restriction. Then the principal symbol of $P^{\mE_{3m,R}}_{R,\beta,\gamma,+}$, which we will denote by $\gamma(P^{\mE_{3m,R}}_{R,\beta,\gamma,+})$, is homotopic through elliptic symbols to $\gamma^{\widetilde \mN_{3m, R}}$. Thus $P^{\mE_{3m,R}}_{R,\beta,\gamma,+}$ is a Fredholm operator. Moreover, by the Atiyah-Singer index theorem \cite{ASI} (cf. \cite[Proposition III.13.8]{LaMi89}), we can calculate the index of $P^{\mE_{3m,R}}_{R,\beta,\gamma,+}$ as follows,
\begin{equation}\label{0.30}
  \begin{aligned}
    {\rm ind}\left(P^{\mE_{3m,R}}_{R,\beta,\gamma,+}\right)
    =\,&{\rm ind}\left(\gamma\left(P^{\mE_{3m,R}}_{R,\beta,\gamma,+}\right)\right)
      ={\rm ind}\left(\gamma^{\widetilde \mN_{ 3m,R}}\right) \\
    =\,& \left\langle \widehat{{A}}(TM) f^*\left({\rm ch}(S_+(TS^n(1)))-{\rm ch}(S_-(TS^n(1)))\right), [M] \right\rangle \\
    =\,& {\rm deg}(f) \Big\langle {\rm ch}(S_+(TS^n(1)))-{\rm ch}(S_-(TS^n(1))),[S^n(1)]\Big\rangle \\
    =\,& (-1)^{n\over 2}{\rm deg}(f)\chi(S^n(1)) =2(-1)^{n\over 2}{\rm deg}(f) \neq 0.
  \end{aligned}
\end{equation}
In (\ref{0.30}), for the fourth equality, we use the fact that ${\rm ch}(S_+(TS^n(1)))-{\rm ch}(S_-(TS^n(1)))$ has only the top degree; for the fifth equality, we use~\cite[Proposition III.11.24]{LaMi89}; and the last inequality comes from (\ref{2.2}).

For any $0\leq t\leq 1$, set
\begin{align}\label{0.31}
  P^{\mE_{3m,R}}_{R,\beta,\gamma,+}(t)=P^{\mE_{3m,R}}_{R,\beta,\gamma,+}+\frac{(t-1)w}{\beta}+A^{-\frac{1}{4}}\frac{(1-t)w}{\beta}A^{-\frac{1}{4}}.
\end{align}
Then $P^{\mE_{3m,R}}_{R,\beta,\gamma,+}(t)$ is a smooth family of zeroth order pseudodifferential operators such that the corresponding symbol $\gamma(P^{\mE_{3m,R}}_{R,\beta,\gamma,+}(t))$ is elliptic for $0<t\leq 1$. Thus $P^{\mE_{3m,R}}_{R,\beta,\gamma,+}(t)$ is a continuous family of Fredholm operators for $0<t\leq 1$ with
\begin{equation*}
  P^{\mE_{3m,R}}_{R,\beta,\gamma,+}(1)=P^{\mE_{3m,R}}_{R,\beta,\gamma,+}.
\end{equation*}

Now since $P^{\mE_{3m,R}}_{R,\beta,\gamma,+}(t)$ is continuous on the whole $[0,1]$, if $P^{\mE_{3m,R}}_{R,\beta,\gamma,+}(0)$ is Fredholm and has vanishing index, then we would reach a contradiction with respect to equation (\ref{0.30}), and then complete the proof of Theorem \ref{t1.1}.

Thus we need only to prove the following analog of \cite[Proposition 2.5]{Z17}.

\begin{prop}\label{t0.5}
There exist $\varepsilon,m,R,\beta,\gamma>0$ such that the following identity holds:
\begin{align*}\dim\left({\rm ker}\left(P^{\mE_{3m,R}}_{R,\beta,\gamma,+}(0)\right)\right)=\dim\left({\rm ker}\left(P^{\mE_{3m,R}}_{R,\beta,\gamma,+}(0)^*\right)\right)=0.
\end{align*}
\end{prop}

\begin{proof}
Let $P^{\mE_{3m,R}}_{R,\beta,\gamma}(0): L^2 (\xi)\to L^2 (\xi)$ be given by
\begin{align}\label{0.33}
P^{\mE_{3m,R}}_{R,\beta,\gamma}(0)=A^{-\frac{1}{4}}\widetilde h_R D ^{\mE_{3m,R}}_{\mF\oplus\mF_{1}^\perp,\beta,\gamma}\widetilde h_R A^{-\frac{1}{4}}+A^{-\frac{1}{4}}\frac{W}{\beta}A^{-\frac{1}{4}}.
\end{align}

Since $P^{\mE_{3m,R}}_{R,\beta,\gamma}(0)$ is formally self-adjoint, by (\ref{0.29}) and (\ref{0.31}), we need only to show that 
\begin{align*}\dim\left({\rm ker}\left(P^{\mE_{3m,R}}_{R,\beta,\gamma}(0)\right)\right)=0
\end{align*}
for certain $\varepsilon, m, R, \beta,\gamma>0$. 

Let $s\in {\rm ker}(P^{\mE_{3m,R}}_{R,\beta,\gamma}(0))$. By (\ref{0.33}), one has
\begin{align}\label{0.35}
\left(\widetilde h_R D ^{\mE_{3m,R}}_{\mF\oplus\mF_{1}^\perp,\beta,\gamma}\widetilde h_R+\frac{W}{\beta}\right)A^{-\frac{1}{4}}s=0.
\end{align}

Since $\widetilde h_R=0$ on $\widetilde \mN_{ 3m,R}\setminus \widehat \mM_{\mH_{3m},R}$, while $W$ is invertible on $\widetilde \mN_{3m, R}\setminus \widehat \mM_{\mH_{3m},R}$, by (\ref{0.35}), one has
\begin{align*}A^{-\frac{1}{4}}s=0\ \ {\rm on}\ \ \widetilde \mN_{ 3m,R}\setminus \widehat \mM_{\mH_{3m},R}.
\end{align*}

Write on $\widehat \mM_{\mH_{3m},R}$ that
\begin{align}\label{0.37}
A^{-\frac{1}{4}}s=s_1+s_2,
\end{align}
with $s_1\in L^2 \left(S_{\beta,\gamma} \left (\mF\oplus\mF_{1}^\perp\right)\widehat\otimes
\Lambda^*\left(\mF_2^\perp \right)\widehat \otimes \mE_{3m,R}\right)$ and $s_2\in L^2 (Q\oplus Q)$.

By (\ref{0.27}), (\ref{0.35}) and (\ref{0.37}), one has
\begin{align*}s_2=0,
\end{align*}
while
\begin{align}\label{0.39}
\bigg(\widetilde h_R D ^{\mE_{3m,R}}_{\mF\oplus\mF_{1}^\perp,\beta,\gamma}\widetilde h_R+\frac{ \widehat c(\widetilde\sigma)}{\beta}+\frac{\varepsilon\widehat{f}^*_{3m,R}V}{\beta}\bigg)s_1=0.
\end{align}

We need to show that (\ref{0.39}) implies $s_1=0$.  

As in (\ref{1.26e}), one has 
\begin{multline}\label{0.40}
 \sqrt{2}\bigg\|\bigg(\widetilde h_R  D ^{\mE_{3m,R}}_{\mF\oplus\mF_{1}^\perp,\beta,\gamma} \widetilde h_R 
+\frac{\widehat c(\widetilde\sigma)}{\beta}
+\frac{\varepsilon\widehat{f}^*_{3m,R}V}{\beta}\bigg)s_1\bigg\|
\\
\geq \bigg\|\bigg(\widetilde h_R  D ^{\mE_{3m,R}}_{\mF\oplus\mF_{1}^\perp,\beta,\gamma}\widetilde h_R 
+\frac{\widehat c(\widetilde\sigma)}{\beta}
+\frac{\varepsilon\widehat{f}^*_{3m,R}V}{\beta}\bigg)(\varphi_{m,1}s_1)\bigg\|
\\
\qquad+\bigg\| \bigg(\widetilde h_R   D ^{\mE_{3m,R}}_{\mF\oplus\mF_{1}^\perp,\beta,\gamma}\widetilde h_R  
+\frac{\widehat c(\widetilde\sigma)}{\beta}
+\frac{\varepsilon\widehat{f}^*_{3m,R}V}{\beta}\bigg)(\varphi_{m,2}s_1)\bigg\| \\
-\|c_{\beta,\gamma}({\rm d}\varphi_{m,1})s_1\|
-\|c_{\beta,\gamma}({\rm d}\varphi_{m,2})s_1\|.
\end{multline}

By proceeding as in the proof of (\ref{2.33}), one gets
\begin{equation}\label{0.41}
  \bigg\|\bigg(\widetilde h_R D ^{\mE_{3m,R}}_{\mF\oplus\mF_{1}^\perp,\beta,\gamma}
      \widetilde h_R 
      +\frac{\widehat c\left(\widetilde\sigma\right)}{\beta}
      +\frac{\varepsilon\widehat{f}^*_{3m,R}V}{\beta}\bigg)(\varphi_{m,2}s_1)\bigg\|^2
  \geq 
  \frac{\varepsilon^2\delta}{\beta^2} \left\|  \varphi_{m,2}s_1 \right\|^2.
\end{equation}

On the other hand, we can use Lemma \ref{t0.4} and proceed as in \cite[p. 1062]{Z17}.
Especially, we need to choose the parameters in the following way.
Firstly, we fix a small enough $\varepsilon>0$.
Then we choose $m>0$ sufficiently large.
The next step is taking $R>0$ sufficiently large.
Finally, we choose a small enough $\gamma$.
With these parameters fixed, we can find a constant $c_1 > 0$ such that for any $\beta$ sufficiently small, the following inequality holds
\begin{align}\label{0.42}
\bigg\|\bigg(\widetilde h_R D ^{\mE_{3m,R}}_{\mF\oplus\mF_{1}^\perp,\beta,\gamma}
 \widetilde h_R 
+\frac{\widehat c\left(\widetilde\sigma\right)}{\beta}
+\frac{\varepsilon\widehat{f}^*_{3m,R}V}{\beta}\bigg)(\varphi_{m,1}s_1)\bigg\| 
\geq 
\frac{c_1}{\beta} \|  \varphi_{m,1} s_1 \| .
\end{align}

From {(\ref{2.19})} and (\ref{0.40})-(\ref{0.42}), by using the same order to choose parameters as in (\ref{0.42}), one finds that there exist $c_2>0$, $\varepsilon>0$, $m>0$, $R>0$ and $\gamma >0$ such that when $\beta$ are sufficiently small, one has 
\begin{align*}\label{0.43}
\bigg\|\bigg(\widetilde h_R D ^{\mE_{3m,R}}_{\mF\oplus\mF_{1}^\perp,\beta,\gamma}
 \widetilde h_R 
+\frac{\widehat{c}(\widetilde{\sigma})}{\beta}+\frac{\varepsilon\widehat{f}^*_{3m,R}V}{\beta}
\bigg)  s_1 \bigg\| 
\geq 
\frac{c_2}{\beta} \|  s_1 \| ,
\end{align*}
which implies, via (\ref{0.39}), $s_1=0$. 
\end{proof}

\subsection{The general case}
\label{sub:locally-const-near}
Till now, we only deal with the case that $f$ is constant near infinity.
To handle the general case that $f$ is locally constant near infinity as stated in Theorem~\ref{t1.1}, we need some modification for the proof.
Note that the most arguments in Subsection~\ref{s1.4} are independent of whether $f$ is constant or locally constant near infinity.
Therefore, what we need to do mainly is to establish Lemma~\ref{t0.4} in this general case.

We use the same notation in this subsection as in Subsection~\ref{s1.1}.
But, now, outside the compact subset $K \subseteq M$, $f$ is locally constant.
Note that the number of connected components of $M\setminus M_{H_{3m}}$ is finite at most.
Let $\{Y_k\}_{k=1}^{l}$ be the connected components of $M \setminus M_{H_{3m}}$.
Assume $f(Y_k)=p_k \in S^n(1)$, $k= 1,\dots, l$. 

Since outside $K$, $f$ may take several values now, we need to modify the construction of the endomorphism $V$ (or $v$) a little.
Due to $\mathrm{deg}(f) \neq 0$, $f$ is a surjective map.
Thus, we can choose a regular value $\mathfrak{p}\neq p_1,\dots, p_l$ of $f$.
Now, we can choose
\begin{equation*}
  v=c(X):S_+(TS^n(1))\to S_-(TS^n(1)),
\end{equation*}
where $X$ is a smooth vector field such that $|X|>0$ on $S^n(1) \setminus \{\mathfrak{p}\}$.
Then, $v$ is invertible over $S^n(1) \setminus \{\mathfrak{p}\}$ and define $V = v + v^*$ as before.

The main difficulty about this general case is how to extend $f$ from $M_{H_{3m}}$ to $\widehat{M}_{H_{3m}}$.
To deal with this problem, we choose a point $p_0\in S^n(1) \setminus \{\mathfrak{p}\}$ and for $k= 1,\dots,l$, pick a curve $\xi_k(\tau), \tau\in [0,1]$, connecting $p_k$ and $p_0$ such that $\xi_k(\tau)\cap \{\mathfrak{p}\}=\emptyset$, $\tau\in [0,1]$.
Then, for $(y,\tau)\in H_{3m}\times [-1,2]$, $k=1,\dots,l$, we define
\begin{align*}
f(y,\tau)=
\begin{cases}
p_k,\ (y,\tau)\in (Y_k\cap H_{3m})\times [-1,0],\\
\xi_k(\tau),\ (y,\tau)\in (Y_k\cap H_{3m})\times [0,1],\\
p_0,\ (y,\tau)\in (Y_k\cap H_{3m})\times [1,2].
\end{cases}
\end{align*}
Note that some points of $\{p_k\}_{k=1}^l$ may coincide.

Recall that $H_{3m}\times [-1, -1+ \varepsilon')$ can be identified with a neighborhood $U$ of $H_{3m}$ in $M_{H_{3m}}$. Under such an identification, the above $f(y,\tau)$ coincides with $f$ on $U$. Thus, $f$ can be extended to a map on $M_{H_{3m}} \cup (H_{3m}\times [-1,2])$ via $f(y,\tau)$ and furthermore, such an extended map can be extended to $\widehat{M}_{H_{3m}}$ by setting $f(M'_{H_{3m}})={p_0}$.
Denote such a map on $\widehat{M}_{H_{3m}}$ by $f_l$.
We will use $f_l$ to substitute the role played by $f$ in Subsection~\ref{s1.2} \& \ref{s1.3}.
Especially, we note that $f_l$ has the following properties\footnote{In general, $f_l$ is only area decreasing along $F$ on $M_{H_{3m}}$ rather than on the whole $\widehat{M}_{H_{3m}}$ (even if in the case that $F$ can be extended to a foliation on $\widehat{M}_{H_{3m}}$). But this is enough for our purpose.}:
\begin{equation*}
  \supp(\dd f_l) \subseteq \supp(\dd f) \cup (H_{3m} \times [0,1]),\quad {\rm deg}(f_l)={\rm deg}(f)\neq 0,
\end{equation*}
and there exists $\delta > 0$ such that
\begin{equation*}
\left  (f_l^*V\right)^2 \ge \delta\text{ on } \widehat{M}_{H_{3m}} \setminus \supp(\dd f).
\end{equation*}
Hence, the following counterpart of (\ref{2.11a}) (or (\ref{2.9})) holds
\begin{equation*}
  \label{eq:fv-bd}
 \big (\widehat{f}_{l,3m,R}^*V\big)^2\geq \delta \text{ on }\ \widehat{\mM}_{\mH_{3m},R}\setminus \pi^{-1}(\supp(\dd f)).
\end{equation*}

Among the estimates in Subsection~\ref{s1.3}, the first one that needs to be modified is (\ref{2.27a}), which is changed to
\begin{align}\label{eq:n2.27a}
\bigg[D ^{\mE_{3m,R}}_{\mF\oplus\mF_{1}^\perp,\beta,\gamma},{{\varepsilon \widehat{f}^*_{{l,}3m,R}V}\over{\beta}}\bigg]= 0.
\end{align}
on $\widehat{\mM}_{\mH_{3m},R}\setminus \left(\pi^{-1}(\supp(\dd f))\cup (\mH_{3m,R}\times [0,1])\right)$. Due to our defintion of the metric on $\widehat{\mM}_{\mH_{3m},R}$, (\ref{eq:metric}), the metric on $\mH_{3m,R} \times [0,1]$ is independent of $\beta,\gamma$.
Therefore, we have
\begin{align}
  \label{eq:n2.27b}
 \bigg[D ^{\mE_{3m,R}}_{\mF\oplus\mF_{1}^\perp,\beta,\gamma},{{\varepsilon \widehat{f}^*_{{l,}3m,R}V}\over{\beta}}\bigg] = O_{m,R}\left({\varepsilon\over \beta}\right) \text{ on } \mH_{3m,R} \times [0,1].
\end{align}

By (\ref{eq:n2.27a}) and (\ref{eq:n2.27b}), (\ref{2.33}) is changed to
\begin{multline*}\bigg\|\bigg(D ^{\mE_{3m,R}}_{\mF\oplus\mF_{1}^\perp,\beta,\gamma}+\frac{\widehat c(\widetilde\sigma)}{\beta} + \frac{\varepsilon \widehat{f}^*_{{l,}3m,R}V}{\beta}\bigg)(\varphi_{m,2}s)\bigg\|^2 \geq {{\delta\varepsilon^2}\over{\beta^2}}\|\varphi_{m,2}s\|^2\\
  + O_{m,R}\left({\varepsilon\over\beta}\right)\big\|{\varphi_{m,2}}s\big\|^2_{\mH_{3m,R} \times [0,1]}.
\end{multline*}
Consequently, (\ref{2.34}) is changed to
\begin{multline*}\sum_{j=1}^2\bigg\|\bigg(D ^{\mE_{3m,R}}_{\mF\oplus\mF_{1}^\perp,\beta,\gamma}+\frac{\widehat c(\widetilde\sigma)}{\beta}
+\frac{\varepsilon \widehat{f}^*_{{l,}3m,R}V}{\beta}\bigg)(\varphi_{m,j}s)\bigg\|^2 \geq \min\left\{{\kappa\over 8},\delta\varepsilon^2\right\}{\|s\|^2\over \beta^2}\\
+ O_{R}\left(\varepsilon\gamma\over \beta\right)\|s\|^2_{\pi^{-1}({\rm Supp}(\dd f))}+O_{m,R}\left(\gamma^2\over \beta^2\right)\|\varphi_{m,1}s\|^2 + O_{m,R}\left(1\over \beta\right)\|\varphi_{m,1}s\|^2 \\
+ O_{m}\left({1\over \beta^2 R}\right) \|\varphi_{m,1}s\|^2 + O_{m,R}\left({\varepsilon\over\beta}\right) \big\|{\varphi_{m,2}}s\big\|^2_{\mH_{3m,R} \times [0,1]}.
\end{multline*}

Using these updated estimates, we can proceed as in Subsection~\ref{s1.3} to obtain Lemma~\ref{t0.4}.

Now, the only point in Subsection~\ref{s1.4} that we need to change is to replace (\ref{0.41}) with the following estimate,
\begin{multline*}\label{eq:n0.41}
\bigg\|\bigg(\widetilde h_R D ^{\mE_{3m,R}}_{\mF\oplus\mF_{1}^\perp,\beta,\gamma}
 \widetilde h_R 
+\frac{\widehat c\left(\widetilde\sigma\right)}{\beta}
+\frac{\varepsilon\widehat{f}^*_{{l,}3m,R}V}{\beta}\bigg)(\varphi_{m,2}s_1)\bigg\|^2\\
\geq 
\frac{\varepsilon^2\delta}{\beta^2} \left\|  \varphi_{m,2}s_1 \right\|^2+O_{m,R}\left({\varepsilon\over \beta}\right)\big\|\varphi_{m,2}s_1\big\|_{\mH_{3m,R} \times [0,1]}^2.
 \end{multline*}
Then the proof of Theorem~\ref{t1.1} in this general case is completed.

\section{Proof of Theorem \ref{t1.1}: the odd dimensional case}\label{sec:odd}

\setcounter{equation}{0}

In this section, we prove Theorem \ref{t1.1} for the odd dimensional case.

Let $M$ be an odd dimensional noncompact manifold of dimension $n$ carrying the complete Riemannian metric $g^{TM}$. Let $F\subseteq TM$ be an integrable subbundle of $TM$. We will use the notation in Section \ref{s2}. Let $f:M\to S^{n}(1)$ be a smooth map which is area decreasing along $F$, locally constant near infinity and of non-zero degree. {Let $g^F=g^{TM}|_{F}$ be the restricted metric on $F$ and} let $k^F$ be the associated leafwise scalar curvature.
As in the even dimensional case, we assume that $TM$ is spin.
We still argue by contradiction, that is, we assume that (\ref{eq:kf-lb}) holds.

For any $r>1$, let $S^1(r)$ be the round circle of radius $r$, with the canonical metric {$\dd \theta^2$}. Let $M\times S^1(r)$ be the complete Riemannian manifold of the product metric $g^{TM}\oplus { \dd \theta^2}$. Following \cite{LL}, we consider the chain of maps
\begin{equation*}
  M\times S^1(r)\xrightarrow{f\times {1\over r}{\rm id}}S^{n}(1)\times S^1(1)\xrightarrow{{h}}S^{n}(1)\wedge S^1 (1)\cong S^{n+1}(1),
\end{equation*}
where $f\times {1\over r}{\rm id}$ is defined as
\begin{equation*}
  \Big(f\times {1\over r}{\rm id}\Big)(x,\theta)=\Big(f(x), {\theta\over r}\Big),\;(x,\theta)\in M\times S^1(r),
\end{equation*}
and ${h}$ is a suspension map of degree one such that $|{\rm d}{h}|\leq 1$. Let $f_r={h}\circ \left(f\times {1\over r}{\rm id}\right)$ denote the composition. Then one has 
\begin{equation*}
  {\rm deg}(f_r)={\rm deg}(f)\neq 0.
\end{equation*}

As in Subsection \ref{s1.2}, we can construct the manifold $\widehat{\mM}_{\mH_{3m},R}$ with the Riemannian metric $g^{T\widehat{\mM}_{\mH_{3m},R}}$ as (\ref{eq:metric}). Set
\begin{equation*}
  \widehat{\mM}_{\mH_{3m},R,r}=\widehat{\mM}_{\mH_{3m},R}\times S^1(r)
\end{equation*}
and the metric on it to be $g^{T\widehat{\mM}_{\mH_{3m},R}}\oplus {{\beta^2 \dd \theta^2}}$. Then
\begin{equation*}
  T\widehat \mM_{\mH_{3m},R,r}= \left(\left(\mF\oplus TS^1(r)\right)\oplus\mF^\perp_{1} \right)\oplus \mF^\perp_{2}\ \  {\rm on}\ \ \widehat \mM_{\mH_{3m},R,r}.
\end{equation*}

Let
\begin{equation*}
  \widehat{f}_{3m,R,r}={h}\circ \Big(\widehat{f}_{3m,R}\times {1\over r}{\rm id}\Big):\widehat{\mM}_{\mH_{3m},R,r}\to S^{n+1}(1)
\end{equation*}
and let $S(TS^{n+1}(1))=S_+(TS^{n+1}(1))\oplus S_-(TS^{n+1}(1))$ be the spinor bundle of $S^{n+1}(1)$.
The pull-back bundle of $S_{\pm}(TS^{n+1}(1))$ via $\widehat{f}_{3m,R,r}$ is denoted by
\begin{equation*}
  \Big(\mE_{3m,R,r,\pm},g^{\mE_{3m,R,r,\pm}},\nabla^{\mE_{3m,R,r,\pm}}\Big)
  = \widehat f_{3m,R,r}^*\left (S_{\pm}(TS^{n+1}(1)),g^{S_{\pm}(TS^{n+1}(1))},\nabla^{S_{\pm}(TS^{n+1}(1))}\right).
\end{equation*}
Then
\begin{equation*}
  \mE_{3m,R,r}=\mE_{3m,R,r,+}\oplus \mE_{3m,R,r,-}
\end{equation*}
is a ${\mathbb{Z}}_2$-graded Hermitian vector bundle over $\widehat \mM_{\mH_{3m},R,r}$. Let $D ^{\mE_{3m,R,r}}_{\mF\oplus\mF_{1}^\perp,\beta,\gamma}$ be the twisted sub-Dirac operator on $\widehat \mM_{\mH_{3m},R,r}$ defined as in (\ref{0.11}).

As (\ref{0.15}), for $\varepsilon>0$, we consider the operator
\begin{align}\label{0.15a}
 D ^{\mE_{3m,R,r}}_{\mF\oplus\mF_{1}^\perp,\beta,\gamma}
+\frac{\widehat c(\widetilde\sigma)}{\beta}
+\frac{\varepsilon \widehat{f}_{3m,R,r}^*V}{\beta},
\end{align}
where $V:S(TS^{n+1}(1))\to S(TS^{n+1}(1))$ is the operator defined {in the same way as the operator $V$ appeared in Subsection~\ref{sub:locally-const-near} except that we should use $f_r$ to replace $f$. The map $\widehat{f}_{3m,R,r}$ in the above formula may be written as $\widehat{f}_{l,3m,R,r}$ in view of the symbols used in Subsection~\ref{sub:locally-const-near}. We omit the $l$ subscript to simplify the symbol a little. As before, there exists $\delta'>0$ such that 
  \begin{align}
    \label{eq:o-f-e}
    \left(\widehat{f}^*_{3m,R,r}V\right)^2\geq \delta'\ {\rm on}\ \widehat{\mM}_{\mH_{3m},R,r}\setminus \pi^{-1}({\rm Supp}(\dd f))\times S^1(r).
  \end{align}

  Let $f_{q+1}$ be an orthonormal basis of $(TS^{1}(r), {\dd \theta^2})$.
  Then proceeding as \cite[p. 68]{LL}, (\ref{2.25}) is replaced by
  \begin{multline}\label{2.25a}
    \bigg({1\over{2\beta^2}}\sum_{i,j=1}^{q+1}R^{\mE_{3m,R,r}}(f_i,f_j)c_{\beta,\gamma}(\beta^{-1}f_i)c_{\beta,\gamma}(\beta^{-1}f_j)s,s\bigg)_{\pi^{-1}({\rm Supp}(\dd f))\times S^1(r)}\\
    \geq {-{{q(q-1)}\over{4\beta^2}}}\big\|s\big\|^2_{\pi^{-1}({\rm Supp}(\dd f))\times S^1(r)}+O\left(1\over \beta^2 r\right)\big\|s\big\|^2_{\pi^{-1}({\rm Supp}(\dd f))\times S^1(r)},
  \end{multline}
for any $s\in \Gamma \left (S_{\beta,\gamma}   \left(\left(\mF\oplus TS^1(r)\right)\oplus \mF_{1}^\perp\right )\widehat\otimes
\Lambda^* \left(\mF_{2}^\perp \right)\widehat \otimes \mE_{3m,R,r} \right )$ supported in the interior of $\widehat\mM_{\mH_{3m},R,r}$.

Since $r>1$, proceeding as (\ref{eq:d-v}) and (\ref{eq:n-v}), on $\pi^{-1}({\rm Supp}(\dd f))\times S^1(r)$, one has
\begin{align}\label{2.26b}
  \bigg[D ^{\mE_{3m,R,r}}_{\mF\oplus\mF_{1}^\perp,\beta,\gamma},{{\varepsilon \widehat{f}^*_{3m,R,r}V}\over{\beta}}\bigg]=O\left({\varepsilon\over \beta^2 }\right)+O_{R}\left({\varepsilon\gamma\over \beta}\right).
\end{align}

On $\big(\widehat{\mM}_{\mH_{3m},R}\setminus \pi^{-1}({\rm Supp}(\dd f))\big)\times S^1(r)$, by (\ref{eq:n2.27a}), (\ref{eq:n2.27b}) and proceeding as (\ref{eq:d-v}), one has
\begin{align}
  \label{eq:o-f-e1}
  \bigg[D ^{\mE_{3m,R,r}}_{\mF\oplus\mF_{1}^\perp,\beta,\gamma},{{\varepsilon \widehat{f}^*_{3m,R,r}V}\over{\beta}}\bigg]=O_{m,R}\left(\varepsilon\over \beta\right)+O_{m,R}\left({\varepsilon\over \beta^2 r }\right).
\end{align}

On $\pi^{-1}(B_{2m}\setminus {\rm Supp}(\dd f))\times S^1(r)$, we also have
\begin{align}\label{725}
R^{{\mE_{3m,R,r}}}=0,\quad \inf(k^\mF)\geq 0.
\end{align}

Define $\varphi_{m,i,r}:\widehat{\mM}_{\mH_{3m},R,r}\to [0,1]$, $i=1,2$, to be the pull-back of $\varphi_{m,i}$ via the projection $\widehat{\mM}_{\mH_{3m},R,r}$ to $\widehat{\mM}_{\mH_{3m},R}$.
Then, we can argue as in the proof of Lemma \ref{t0.4} by using (\ref{eq:o-f-e})--(\ref{725}).
The difference is that after fixing the parameters $\varepsilon, m, R, \gamma$ in the order given before, we further need to choose $r> 1$ sufficiently large.
As a result, for $\beta$ small enough, the analog of Lemma \ref{t0.4} still holds for the operator (\ref{0.15a}).

Similarly as (\ref{0.29}), we can define the pseudodifferential operator
$P^{\mE_{3m,R,r}}_{R,\beta,\gamma}$ and we also have
\begin{align}
  \label{eq:o-ind}
{\rm ind}\big(P^{\mE_{3m,R,r}}_{R,\beta,\gamma,+}\big)=2(-1)^{n+1\over 2}{\rm deg}(f_r)\neq 0.
\end{align}
On the other hand, proceeding as the proof of Proposition \ref{t0.5}, by taking the parameters in the order $\varepsilon, m, R, \gamma, r, \beta$, the analog of Proposition \ref{t0.5} still holds for the operator $P^{\mE_{3m,R,r}}_{R,\beta,\gamma,+}$, which contradicts to (\ref{eq:o-ind}).
The proof for the odd dimensional case is finished.

\section{Proof of Theorem \ref{th6}}
\label{sec:th6}

In this section, we prove Theorem \ref{th6}. 

{Let $(M,F)$ be a foliated manifold. Let $g^{TM}$ be a complete Riemannian metric on $TM$ }and let $g^{F}=g^{TM}|_{F}$ be the restricted metric on $F$. Let $k^F$ be the associated leafwise scalar curvature on $F$. We assume that the Riemannian metric $g^{TM}$ is $\Lambda^2$-enlargeable along $F$.

We assume that $\dim M$ is even. If $\dim M$ is odd, one may consider $M\times S^1$ and use the method in Section~\ref{sec:odd}.

We still argue by contradiction. Assume there is $\delta>0$ such that 
\begin{align}\label{aa1}
k^F\geq \delta\ \ {\rm over}\ \ M.
\end{align}

Let $F^\perp$ be the orthogonal complement to $F$, i.e., we have the orthogonal splitting 
 \begin{align}\label{aa5.1}
TM=F\oplus F^\perp,\ \ \ g^{TM}=g^F\oplus g^{F^\perp}.
\end{align}

By the definition, for any $\epsilon>0$, there exists a covering $\pi_\epsilon: M_\epsilon\to M$ such that either $M_\epsilon$ or $F_\epsilon$ (the lifted foliation of $F$ in $M_\epsilon$) is spin and a smooth map $f_\epsilon: M_\epsilon\to S^{\dim M}(1)$ which is $(\epsilon,\Lambda^2)$-contracting along $F_\epsilon$ (with respect to the lifted metric of $g^{TM}$), constant outside a compact subset $K_\epsilon$ and of non-zero degree.

We will give the proof for the $M_\epsilon$ spin case, since one can prove the $F_\epsilon$ spin case by combining the $M_\epsilon$ spin case and the argument in \cite[\S 2.5]{Z17}.

Let $g^{TM_\epsilon}=\pi^*_\epsilon g^{TM}$ be the lifted metric of $g^{TM}$ and let $g^{F_\epsilon}=\pi^*_\epsilon g^F$ be the lifted Euclidean metric on $F_\epsilon$. The splitting (\ref{aa5.1}) lifts canonically to a splitting 
 \begin{align*}\label{aa5.2}
TM_\epsilon=F_\epsilon\oplus F_\epsilon^\perp,\ \ \ g^{TM_\epsilon}=g^{F_\epsilon}\oplus g^{F_\epsilon^\perp}.
\end{align*}

If both $M$ and $M_\epsilon$ are compact, by \cite[Section 1.1]{Z19}, one gets a contradiction easily.

In the following, we assume that $M_\epsilon$ is noncompact.

For $(M_\epsilon, F_\epsilon)$ equipped with the metrics $(g^{TM_\epsilon}, g^{F_\epsilon})$ and the smooth map
\begin{equation*}
  f_\epsilon:M_\epsilon \allowbreak\to S^{\dim M}(1),
\end{equation*}
one can follow the steps shown in Section \ref{s2}. We will use $\epsilon$ to denote the corresponding objects in this case.

Let $\widetilde{\pi}_\epsilon:\mM_\epsilon \to M_\epsilon$ be the Connes fibration. Set
\begin{equation*}
  k^{F_\epsilon}=\pi^*_\epsilon\left(k^F\right),\quad k^{\mF_\epsilon}=\widetilde{\pi}^*_\epsilon\left(\pi^*_\epsilon(k^F)\right).
\end{equation*}

 With these settings, as in Section~\ref{s2}, the key to find a contradiction is to prove an analog of Lemma~\ref{t0.4} for the operator
  \begin{equation*}
    D^{\mE_{\epsilon,3m,R}}_{\mF_\epsilon\oplus\mF_{\epsilon,1}^\perp,\beta,\gamma} +\frac{\widehat c(\widetilde\sigma_\epsilon)}{\beta} +\frac{\varepsilon \widehat{f}_{\epsilon,3m,R}^*V_\epsilon}{\beta}.
  \end{equation*}
  To show such an analog, after checking the proof of Lemma~\ref{t0.4}, the we only need to prove estimates to replace (\ref{2.25}) and (\ref{g2.24}).

 Let $\nabla^{S(TS^{\dim M}(1))}$ be the canonical connection on the spinor bundle of $S^{\dim M}(1)$. Let $R^{S(TS^{\dim M}(1))}$ be the curvature tensor of the connection. Set\footnote{Here we need not use the precise estimate in \cite{LL}.}
\begin{align}\label{cc4.5}
C_1=\sup_{p\in S^{\dim M}(1)}\Big|R_p^{S\left(TS^{\dim M}(1)\right)}\Big|.
\end{align}

Choose a local frame of $\widehat{\mM}_{\mH_{\epsilon,3m},R}$ as in (\ref{eq:loc-fr}). By the $(\epsilon,\Lambda^2)$-contracting property of $f_\epsilon$, we have the following pointwise estimate,

 \begin{multline}\label{aa6}
   \Big| \Big({\frac{1}{2\beta^2}}\sum^{q}_{i,j=1}R^{\mE_{\epsilon,3m,R}}(f_i,f_j)c_{\beta,\gamma}(\beta^{-1}f_i)c_{\beta,\gamma}(\beta^{-1}f_j)s,s\Big)(x)\Big|\\
   \begin{aligned}
   =&\Big | \Big({\frac{1}{2\beta^2}}\sum^{q}_{i,j=1}\widehat{f}^*_{\epsilon,3m,R}\Big(R^{S}\big(\widehat{f}_{\epsilon,3m,R,*}f_i,\widehat{f}_{\epsilon,3m,R,*}f_j\big)\Big)c_{\beta,\gamma}(\beta^{-1}f_i)c_{\beta,\gamma}(\beta^{-1}f_j)s,s\Big)(x)\Big|\\
   =&\Big | \Big({\frac{1}{2\beta^2}}\sum^{q}_{i,j=1}\widehat{f}^*_{\epsilon,3m,R}\Big(R^{S}\big(\widehat{f}_{\epsilon,3m,R,*}(f_i\wedge f_j)\big)\Big)c_{\beta,\gamma}(\beta^{-1}f_i)c_{\beta,\gamma}(\beta^{-1}f_j)s,s\Big)(x)\Big|
   \end{aligned}
   \\   
   \leq  \frac{1 }{2\beta^2} q (q-1)C_1 \epsilon |s|^2(x),
 \end{multline}
where $R^S$ is the shorthand for $R^{S(TS^{\dim M}(1))}$.

 Now, we choose
  \begin{equation*}
    \epsilon = \frac{\delta}{4C_1q^2}.
  \end{equation*} 
 Then, using the notations in Section \ref{s2}, by (\ref{aa1}) and (\ref{aa6}), we have
\begin{multline*}
  \label{eq:k-pos}
  \bigg({\frac{1}{2\beta^2}}\sum^{q}_{i,j=1}R^{\mE_{\epsilon,3m,R}}(f_i,f_j)c_{\beta,\gamma}(\beta^{-1}f_i)c_{\beta,\gamma}(\beta^{-1}f_j)s,s\bigg)_{\widetilde{\pi}^{-1}_\epsilon(K_\epsilon)}\\
+ \left(\frac{k^{\mF_\epsilon}}{4\beta^2}s,s\right)_{{\widetilde{\pi}_\epsilon}^{-1}(K_\epsilon)} \geq {\frac{\delta}{8\beta^2}}\|s\|^2_{\widetilde{\pi}^{-1}_\epsilon(K_\epsilon)},
\end{multline*}
which can be used to replace (\ref{2.25}) and (\ref{g2.24}). The remaining argument to get a contradiction follows from the same method used in Section~\ref{s2}.

\section{Proof of Theorem \ref{t5}}
\label{sec:th5}

In this section, we prove Theorem \ref{t5}.

{Let $(M,F)$ be a foliated manifold.} We assume that $M$ is $\Lambda^2$-enlargeable along $F$. Let $g^{TM}$ be a complete {Riemannian} metric on $TM$ and $g^{F}=g^{TM}|_{F}$ be the restricted metric on $F$. Let $k^F$ be the associated leafwise scalar curvature on $F$.

As before, we argue by contradiction. Assume that 
\begin{equation*}
  k^F>0\ \ {\rm over}\ \ M.
\end{equation*}

Let $F^\perp$ be the orthogonal complement to $F$, i.e., we have the orthogonal splitting 
 \begin{align}\label{5.1}
   TM=F\oplus F^\perp,\ \ \ g^{TM}=g^F\oplus g^{F^\perp}.
 \end{align}

Inspired by the proof of \cite[Theorem 6.12]{GL83}, we consider another metric on $TM$ defined by $k^F g^{TM}$. By the definition, for the metric $k^F g^{TM}$ and any $\epsilon>0$, there exist a covering
\begin{equation*}
  \pi_\epsilon: M_\epsilon\to M
\end{equation*}
such that either $M_\epsilon$ or $F_\epsilon$ (the lifted foliation of $F$ in $M_\epsilon$) is spin and a smooth map
\begin{equation*}
  f_\epsilon: M_\epsilon\to S^{\dim M}(1)
\end{equation*}
which is $(\epsilon,\Lambda^2)$-contracting along $F_\epsilon$ for the lifted metric of $k^F g^{TM}$, constant outside a compact subset and of non-zero degree.

Let $g^{TM_\epsilon}=\pi^*_\epsilon g^{TM}$ be the lifted metric of $g^{TM}$ and let $g^{F_\epsilon}=\pi^*_\epsilon g^F$ be the lifted Euclidean metric on $F_\epsilon$. The splitting (\ref{5.1}) lifts canonically to a splitting 
 \begin{align*}\label{5.2}
   TM_\epsilon=F_\epsilon\oplus F_\epsilon^\perp,\ \ \ g^{TM_\epsilon}=g^{F_\epsilon}\oplus g^{F_\epsilon^\perp}.
 \end{align*}

We will give the proof for the $M_\epsilon$ spin case, since one can prove the $F_\epsilon$ spin case by combining the $M_\epsilon$ spin case and the argument in \cite[\S 2.5]{Z17}.

We first assume that $\dim M$ is even.

For $(M_\epsilon, F_\epsilon)$ equipped with the metrics $(g^{TM_\epsilon}, g^{F_\epsilon})$ and the smooth map
\begin{equation*}
  f_\epsilon:M_\epsilon \to S^{\dim M}(1),
\end{equation*}
one can follow the steps shown in Section \ref{s2}. We will use $\epsilon$ to denote the corresponding objects in this case.

Let $\widetilde{\pi}_\epsilon:\mM_\epsilon \to M_\epsilon$ be the Connes fibration.
We note that the (deformed or not) metric on $\mM_\epsilon$ is defined as in Subsection~\ref{s1.2}, which means that we use the metric $g^{TM_\epsilon}$ rather than the metric $k^{F_{\epsilon}}g^{TM_\epsilon}$ to define the metric on $\mM_\epsilon$.
Set
\begin{equation*}
  k^{F_\epsilon}=\pi^*_\epsilon\left(k^F\right) \quad k^{\mF_\epsilon}=\widetilde{\pi}^*_\epsilon\left(\pi^*_\epsilon\left(k^F\right)\right).
\end{equation*}

With these settings, as in Section~\ref{s2}, the key to find a contradiction is to prove an analog of Lemma~\ref{t0.4} for the operator
\begin{equation*}
  D^{\mE_{\epsilon,3m,R}}_{\mF_\epsilon\oplus\mF_{\epsilon,1}^\perp,\beta,\gamma} +\frac{\widehat c(\widetilde\sigma_\epsilon)}{\beta} +\frac{\varepsilon \widehat{f}_{\epsilon,3m,R}^*V_\epsilon}{\beta}.
\end{equation*}
As in the proof Theorem~\ref{th6}, to show such an analog, we only need to prove estimates to replace (\ref{2.25}) and (\ref{g2.24}).

Choose a local frame of $T{\mM}_{\epsilon,\mH_{3m},R}$ as in (\ref{eq:loc-fr}). By the $(\epsilon,\Lambda^2)$-contracting property of $f_\epsilon$ for the metric $k^{F_\epsilon} g^{TM_\epsilon}$ and $|f_i\wedge f_j|_{g^{TM_\epsilon}}=1$, $i\neq j$, we have
\begin{align}\label{ne5.4}
\left|f_{\epsilon,*}(f_i\wedge f_j)\right|\leq \epsilon \left|f_i\wedge f_j\right|_{k^{F_\epsilon} g^{TM_\epsilon}}=\epsilon k^{F_\epsilon},\ i\neq j.
\end{align}

Then for $x \in \widetilde{\pi}_\epsilon^{-1}(\supp(\dd f_\epsilon))$, by (\ref{ne5.4}), we have the following pointwise estimate,
\begin{multline}
  \label{eq:est-kf}
  \Big| \Big({\frac{1}{2\beta^2}}\sum^{q}_{i,j=1}R^{\mE_{\epsilon,3m,R}}(f_i,f_j)c_{\beta,\gamma}(\beta^{-1}f_i)c_{\beta,\gamma}(\beta^{-1}f_j)s,s\Big)(x)\Big|\\
  \begin{aligned}
  =& \Big | \Big({\frac{1}{2\beta^2}}\sum^{q}_{i,j=1}\widehat{f}^*_{\epsilon,3m,R}\Big(R^{S}(\widehat{f}_{\epsilon,3m,R,*}f_i,\widehat{f}_{\epsilon,3m,R,*}f_j)\Big)c_{\beta,\gamma}(\beta^{-1}f_i)c_{\beta,\gamma}(\beta^{-1}f_j)s,s\Big)(x)\Big|\\
  =& \Big| \Big({\frac{1}{2\beta^2}}\sum^{q}_{i,j=1}\widehat{f}^*_{\epsilon,3m,R}\Big(R^{S}(\widehat{f}_{\epsilon,3m,R,*}(f_i\wedge f_j))\Big)c_{\beta,\gamma}(\beta^{-1}f_i)c_{\beta,\gamma}(\beta^{-1}f_j)s,s\Big)(x)\Big|
  \end{aligned}\\
  \leq {1\over 2\beta ^2} q (q-1) C_1\epsilon k^{\mF_\epsilon}(x) |s|^2(x),
\end{multline}
where $C_1$ is the constant defined in (\ref{cc4.5}) and as before, $R^S$ is the shorthand for $R^{S(TS^{\dim M}(1))}$.

  Now, we choose
  \begin{equation}
    \label{eq:def-e}
    \epsilon = \frac{1}{4C_1q^2}.
  \end{equation}
  Then $f_\epsilon$ is fixed and $\supp(\dd f_\epsilon)$ is a fixed compact set.
  Hence, we can find $\kappa > 0$ such that
  \begin{equation}
    \label{eq:est-e}
    k^{\mF_\epsilon} \ge \kappa \text{ on }\widetilde{\pi}_\epsilon^{-1}\left(\supp(\dd f_\epsilon)\right).
  \end{equation}
  Therefore, using the notation in Section \ref{s2}, by (\ref{eq:est-kf}), (\ref{eq:def-e}) and (\ref{eq:est-e}), for any point $x\in\allowbreak \widetilde{\pi}_\epsilon^{-1}(\supp(\dd f_\epsilon))$, we have
  \begin{multline*}
    \bigg({1\over{2\beta^2}}\sum_{i,j=1}^{q}R^{\mE_{\epsilon,3m,R}}(f_i,f_j)c_{\beta,\gamma}(\beta^{-1}f_i)c_{\beta,\gamma}(\beta^{-1}f_j)s,s\bigg)(x) + \left(\frac{k^{\mF_\epsilon}}{4\beta^2}s,s\right)(x)\\
    \ge \left(\frac{k^{\mF_\epsilon}}{8\beta^2}s,s\right)(x) \ge \frac{\kappa}{8\beta^2}|s|^2(x),
  \end{multline*}
  which can be used to replace (\ref{2.25}) and (\ref{g2.24}).
  The remaining argument to get a contradiction follows from the same method used in Section~\ref{s2}.

  If $\dim M$ is odd, as in Section \ref{sec:odd}, we can replace $M_\epsilon$ by $M_\epsilon\times S^1(r)$. Consider the composition $ f_{\epsilon,r}$ of the maps
\begin{align*}
M_\epsilon \times S^1(r)\xrightarrow{ f_{\epsilon} \times {1\over r}{\rm id}}S^{\dim M}(1)\times S^1(1)\xrightarrow{{\wedge}} S^{\dim M+1}(1).
\end{align*}
Then this map $f_{\epsilon,r}$ is pointwise $(\max\{\epsilon k^{F_\epsilon},|{\dd}f_{\epsilon}|/r\},\Lambda^2)$-contracting with respect to the metric $g^{TM_\epsilon}\oplus {\dd \theta^2}$.

Fix $\epsilon$ as (\ref{eq:def-e}) and set
\begin{equation*}
  \kappa_0=\min\{k^{F_\epsilon}(x),x\in {\rm Supp}({ \dd f}_{\epsilon})\}.
\end{equation*}
We choose $r$ large enough such that 
 \begin{align*}
 \frac{\sup\left\{|{\dd}f_{\epsilon}|(x),x\in M_\epsilon\right\}}{r}<\epsilon \kappa_0.
 \end{align*}
 Then by combining the method used in the above even dimensional case and the content of Section~\ref{sec:odd}, we can also get a contradiction.

\vspace{\baselineskip}

\noindent{\bf Acknowledgments.}  G. Su and W. Zhang were partially supported by NSFC Grant No. 11931007 and Nankai Zhide Foundation. X. Wang was partially supported by NSFC Grant No. 12101361, the project of Young Scholars of SDU and the fundamental research funds of Shandong University, Grant No.\ 2020GN063. The authors would like to thank the anonymous referee for careful reading and valuable suggestions.


\begin{thebibliography}{99}



 







  


 
  
 



 






 
  
 




 



 







 










 






\providecommand{\url}[1]{\texttt{#1}}
\providecommand{\urlprefix}{URL }
  
\bibitem{ASI}
{M.~F. Atiyah} and {I.~M. Singer}, The index of elliptic
  operators. {I}, Ann. of Math. (2) \textbf{87} (1968), 484--530.

\bibitem{BH19}
{M.-T. Benameur} and {J.~L. Heitsch}, Enlargeability, foliations,
  and positive scalar curvature, Invent. Math. \textbf{215} (2019), no.~1,
  367--382.

\bibitem{BH20}
{M.-T. Benameur} and {J.~L. Heitsch}, Geometric
  noncommutative geometry, Expo. Math. \textbf{39} (2021), no.~3, 454--479.

\bibitem{BL91}
{J.-M.~Bismut} and {G.~Lebeau}, Complex immersions and {Q}uillen
  metrics, Inst. Hautes {\'E}tudes Sci. Publ. Math.  (1991), no.~74, ii+298 pp.
  (1992).

\bibitem{Co86}
{A.~Connes}, Cyclic cohomology and the transverse fundamental class of a
  foliation, in: Geometric methods in operator algebras ({K}yoto, 1983),
  Longman Sci. Tech., Harlow, 1986, \textit{Pitman Res. Notes Math. Ser.},
  volume 123, 52--144.

\bibitem{GW}
{R.~Greene} and {H.~Wu}, {$C^{\infty }$}
  approximations of convex, subharmonic, and plurisubharmonic functions, Ann.
  Sci. \'{E}cole Norm. Sup. (4) \textbf{12} (1979), no.~1, 47--84.

\bibitem{Gr}
{M.~Gromov}, Four lectures on scalar curvature, preprint 2019,
  \urlprefix\url{http://arxiv.org/abs/1908.10612v4}.

\bibitem{GL83}
{M.~Gromov} and {H.~B. Lawson, Jr.}, Positive scalar curvature
  and the {D}irac operator on complete {R}iemannian manifolds, Inst. Hautes
  {\'E}tudes Sci. Publ. Math.  (1983), no.~58, 83--196 (1984).
  
\bibitem{LaMi89}
{H.~B. {Lawson Jr.}} and {M.-L. Michelsohn}, {Spin geometry},
  \textit{Princeton Mathematical Series}, volume~38, Princeton University
  Press, Princeton, NJ, 1989.

\bibitem{L63}
{A.~Lichnerowicz}, Spineurs harmoniques, C. R. Acad. Sci. Paris
  \textbf{257} (1963), 7--9.

\bibitem{LZ01}
{K.~Liu} and {W.~Zhang}, Adiabatic limits and foliations, in:
  Topology, geometry, and algebra: interactions and new directions ({S}tanford,
  {CA}, 1999), Amer. Math. Soc., Providence, RI, 2001, \textit{Contemp. Math.},
  volume 279, 195--208.

\bibitem{LL}
{M.~Llarull}, Sharp estimates and the {D}irac operator, Math. Ann.
  \textbf{310} (1998), no.~1, 55--71.

\bibitem{S}
{G.~Su}, Lower bounds of {L}ipschitz constants on foliations, Math. Z.
  \textbf{293} (2019), no. 1-2, 417--423.

\bibitem{SZ}
{G.~Su} and \textit{W.~Zhang}, Positive scalar curvature on foliations:
  the noncompact case, preprint 2019, \urlprefix\url{http://arxiv.org/abs/1905.12919v1}.

\bibitem{Z17}
{W.~Zhang}, Positive scalar curvature on foliations, Ann. of Math. (2)
  \textbf{185} (2017), no.~3, 1035--1068.

\bibitem{Z19-2}
{W.~Zhang}, Nonnegative scalar curvature and area decreasing maps, SIGMA
  Symmetry Integrability Geom. Methods Appl. \textbf{16} (2020), Paper No. 033,
  7.

\bibitem{Z19}
{W.~Zhang}, Positive scalar curvature on foliations: the enlargeability,
  in: Geometric analysis, Birkh\"{a}user/Springer, Cham, 2020, \textit{Progr.
  Math.}, volume 333, 537--544.
  
\end{thebibliography}
\end{document}